\documentclass{amsart}  

\usepackage[margin=1.5in]{geometry}

\title{A classification of equivariant gerbe connections}

\author{Byungdo Park}
\address{Hausdorff Research Institute for Mathematics (HIM), Poppelsdorfer Allee 45, 53115 Bonn, Germany}
\curraddr{School of Mathematics, Korea Institute for Advanced Study,
85 Hoegiro, Dongdaemun-gu, Seoul, 02455, Republic of Korea}
\email{byungdpark@gmail.com}

\author{Corbett Redden}
\address{Department of Mathematics, LIU Post, Long Island University, 720 Northern Blvd, Brookville, NY 11548, USA}
\email{corbett.redden@liu.edu}

\usepackage{hyperref} 
\hypersetup{ 
colorlinks,%
citecolor=black,%
filecolor=black,%
linkcolor=black,%
urlcolor=black } 
 
\usepackage{stmaryrd}  
\usepackage{amssymb}
\usepackage{mathtools}   
\usepackage{mathrsfs}  
\usepackage{upgreek}  
\usepackage{tikz-cd} 
\usepackage[all]{xy}  
\usepackage{comment}  

\numberwithin{equation}{section}

\theoremstyle{plain}
\newtheorem{thm}[equation]{Theorem}
\newtheorem{prop}[equation]{Proposition}
\newtheorem{lemma}[equation]{Lemma}
\newtheorem{cor}[equation]{Corollary}

\theoremstyle{definition}
\newtheorem{defn}[equation]{Definition}

\newtheorem{example}[equation]{Example}

\newtheorem{rem}[equation]{Remark}

\usepackage{enumitem}

\renewenvironment{itemize}{
  \begin{list}{$\bullet$}{
    \setlength{\leftmargin}{1em}  
} }{
  \end{list}
}

\DeclareFontFamily{OT1}{pzc}{}
\DeclareFontShape{OT1}{pzc}{m}{it}%
              {<-> s * [1.15] pzcmi7t}{}
\DeclareMathAlphabet{\mathpzc}{OT1}{pzc}{m}{it}

\newcommand{\Z}{\mathbb{Z}}
\newcommand{\R}{\mathbb{R}}

\newcommand{\fg}{\mathfrak{g}}
\newcommand{\iso}{\cong}
\renewcommand{\=}{:=}
\newcommand{\pt}{\text{pt}}
\newcommand{\into}{\hookrightarrow}
\newcommand{\cK}{\mathcal{K}}
\newcommand{\cKh}{{\widehat{\mathcal{K}}}}
\newcommand{\wt}{\widetilde}
\newcommand{\fk}{\mathfrak{k}}
\newcommand{\Id}{\operatorname{id}}

\newcommand{\cA}{\mathcal{A}}

\newcommand{\cC}{\mathcal{C}}
\newcommand{\cJ}{\mathcal{J}}
\newcommand{\Aut}{\operatorname{Aut}}

\newcommand\bas{\text{bas}}
\newcommand\Ad{\operatorname{Ad}}
\newcommand\op{\mathrm{op}}

\newcommand\st{\>\> \big| \>\>}
\newcommand\fl{\operatorname{flat}}
\newcommand\Hol{\operatorname{Hol}}
\newcommand\Fun{\operatorname{Fun}}
\newcommand\cl{{\operatorname{cl}}}

\renewcommand{\o}{\circ}

\renewcommand\O{\Omega}

\newcommand{\cF}{\mathcal{F}}

\newcommand{\cat[1]}{\mathsf{#1}}
\newcommand\Man{\cat[Man]}
\newcommand\Set{\cat[Set]}
\newcommand\sSet{\cat[sSet]}
\newcommand\Pre{\cat[Pre]}
\newcommand\Shv{\cat[Shv]}
\newcommand\Gpd{{\cat[Gpd]}}
\newcommand\Ch{{\cat[Ch]}_{\geq 0}}
\newcommand\Top{\cat[Top]}
\newcommand\Ab{\cat[Ab]}
\newcommand\Gpdinfty{\infty\text{-}\cat[Gpd]}
\newcommand\DK{\Gamma} 

\newcommand{\B}{\mathpzc{B}}
\newcommand{\E}{\mathpzc{E}}
\newcommand{\Bn}{\mathpzc{B_{\scriptscriptstyle{\nabla}}\mkern-2mu}}
\newcommand{\En}{\mathpzc{E_{\scriptscriptstyle{\nabla}}\mkern-2mu}}

\newcommand{\BBn}{\mathpzc{B}^2_{\scriptscriptstyle{\nabla}}}

\newcommand\GMan{\text{-}\Man}

\newcommand\cM{\mathpzc{M}}

\newcommand\Bun[1][S^1]{\cat[Bun]_{#1}}
\newcommand\Bunc[1][S^1]{\cat[Bun]_{#1,\nabla}}
\newcommand\GBun[1][S^1]{\text{-}\cat[Bun]_{#1}}
\newcommand\GBunc[1][S^1]{\text{-}\cat[Bun]_{#1,\nabla}}

\newcommand\Grb{\cat[Grb]}
\newcommand\Grbc{\cat[Grb]_{\nabla}}
\newcommand\GGrb{\text{-}\cat[Grb]}
\newcommand\GGrbc{\text{-}\cat[Grb]_{\nabla}}

\newcommand\FP{\operatorname{FP}}
\newcommand\desc{\operatorname{desc}}

\newcommand\GGrbtrivc{\text{-}\cat[Grbtriv]_{\nabla}}

\newcommand\cI{\mathcal{I}}
\newcommand\cIh{\widehat{\mathcal{I}}}
\newcommand\cL{\mathcal{L}}
\newcommand\cLh{\widehat{\mathcal{L}}}

\newcommand\cJh{\widehat{\mathcal{J}}}

\newcommand\cEll{\mathscr{L}}
\newcommand\cEllh{\widehat{\mathscr{L}}}

\newcommand\HG[1][G]{\widehat{H}_{#1}}

\renewcommand{\sslash}{/\!\!/}
\newcommand{\nslash}{ /\!\!/\mkern-4mu_{\scriptscriptstyle{\nabla}}}

\newcommand\DD{\operatorname{DD}}
\newcommand\curv{\operatorname{curv}}

\renewcommand\L{\mathbb L} 
\newcommand\K{\mathpzc{K}} 
\newcommand\Kh{\widehat{\mathpzc{K}}} 

\newcommand{\Tr}{\operatorname{Tr}}

\DeclareMathOperator*{\hocolim}{hocolim}

\DeclareMathOperator*{\holim}{holim}

\newcommand{\al}{\alpha}
\newcommand{\be}{\beta}
\newcommand{\ga}{\gamma}     
\newcommand{\de}{\delta}

\newcommand{\te}{\theta}    \newcommand{\Te}{{\Theta}}
\newcommand{\si}{\sigma}     
\newcommand{\om}{\omega}     \newcommand{\Om}{{\Omega}}
\newcommand{\vph}{\varphi}
      
\newcommand{\tsr}{\otimes}

\newcommand{\srl}{\stackrel}
 \newcommand{\ra}{\rightarrow}  \newcommand{\embed}{\hookrightarrow}

\newcommand{\wtl}{\widetilde}

\newcommand{\na}{\nabla}

\newcommand{\isom}{\cong}   

\newcommand{\bmat}{\left(\begin{array}}  \newcommand{\emat}{\end{array}\right)}
\newcommand{\barr}{\begin{array}}  \newcommand{\earr}{\end{array}}
\newcommand{\bcd}{\begin{CD}}  \newcommand{\ecd}{\end{CD}}
\newcommand{\beq}{\begin{equation}\begin{aligned}}  \newcommand{\eeq}{\end{aligned}\end{equation}}
\newcommand{\beqs}{\begin{equation*}\begin{aligned}}  \newcommand{\eeqs}{\end{aligned}\end{equation*}}

\begin{document}


\begin{abstract}Let $G$ be a compact Lie group acting on a smooth manifold $M$.  In this paper, we consider Meinrenken's $G$-equivariant bundle gerbe connections on $M$ as objects in a 2-groupoid. We prove this 2-category is equivalent to the 2-groupoid of gerbe connections on the differential quotient stack associated to $M$, and isomorphism classes of $G$-equivariant gerbe connections are classified by degree three differential equivariant cohomology.  Finally, we consider the existence and uniqueness of conjugation-equivariant gerbe connections on compact semisimple Lie groups.
\end{abstract}

\maketitle
\thispagestyle{empty} 

\tableofcontents


\section{Introduction}\label{Sec:Intro}

While there exist several models for $S^1$-banded gerbes with connection, working directly with any of them necessarily involves technical details and subtleties, and the complexity grows rapidly should one consider $p$-gerbes for $p>1$.  However, there is a relatively simple classification of gerbe connections {\it up to isomorphism}, which is given by degree 3 differential cohomology (Deligne cohomology) and fits into the pattern
\begin{equation}\label{Eq:LowDegreeDC}  \HG[]^1(M) \iso C^\infty(M,S^1),\quad 
\HG[]^2(M) \iso \Bunc(M)_{/_{\iso}}, \quad  \HG[]^3(M) \iso \Grbc(M)_{/_{\iso}} \, .  \end{equation}
Because the groups $\HG[]^n(M)$ fit into short exact sequences involving differential forms and ordinary cohomology, one can use elementary techniques to get a strong foothold in the gerbe world.

Assume now that $G$ is a compact Lie group acting smoothly on a finite-dimensional manifold $M$.  This paper's aim is to further develop the theory of $G$-equivariant gerbe connections in a manner  analogous to $G\GBunc(M)$, the groupoid of $G$-equivariant $S^1$-bundles on $M$ with $G$-invariant connection.  We do so by considering two specific models. The first is denoted $G\GGrbc(M)$, the 2-groupoid of $G$-equivariant bundle gerbes on $M$ with $G$-equivariant connection.  Its objects, which were defined by Meinrenken \cite{Mein-Grb} and further investigated by Sti\'enon \cite{Sti10} and Tu--Xu \cite{TuXu15}, are differential-geometric in the same way as Murray's bundle gerbes with connection \cite{Mur96}.  We define a 2-categorical structure on $G\GGrbc(M)$ by adapting Waldorf's 2-groupoid structure on $\Grbc(M)$ in a straightforward way \cite{Waldorf-More}.

The second model uses the language of higher stacks, or sheaves of $\infty$-groupoids on the site of manifolds.  For $M\in G\GMan$ we consider the differential quotient stack ${\En G \times_G M}$, which associates to every test manifold $X$ the groupoid of principal $G$-bundles {\it with connection} $(P,\Theta) \in \Bunc[G](X)$, together with $G$-equivariant map $f\colon P\to M$.  In \cite{Redden16a}, the differential equivariant cohomology groups were defined as
\[ \HG^n(M) \= \HG[]^n(\En G \times_G M),\]
suitably interpreted, and it was shown that they fit into short exact sequences \eqref{ses1}--\eqref{ses3}  involving the Borel equivariant cohomology $H_G^*(M;-) \= H^*(EG\times_G M;-)$ and the equivariant differential forms $\O^*_G(M)$ of Cartan--Weil.  Furthermore, there are isomorphisms generalizing \eqref{Eq:LowDegreeDC} in degrees one and two:
\[  \HG^1(M) \iso C^\infty(M;S^1)^G, \quad  \HG^2(M) \iso G\GBunc(M)_{/_{\iso}}.\]
It was further shown in \cite{Redden16b} that there is a natural equivalence of groupoids 
\[ \begin{tikzcd} G\GBunc[K](M) \ar[r, "\iso", shift left] & \Bunc[K](\En G \times_G M)  \ar[l, shift left] \end{tikzcd} \]
for any Lie groups $G, K$.  When combined with the properties of $\HG^3(M)$, this strongly suggests that  $\Grbc(\En G\times_G M)$ should be a valid model for $G$-equivariant gerbe connections.

That hope is confirmed by the two primary results of this paper, Theorems \ref{Thm:IsoClasses} and \ref{Thm:Equiv2Gpd}.  The first theorem states
\begin{equation} G\GGrbc(M)_{/_{\iso}} \overset{\iso}\longrightarrow \HG^3(M), \end{equation}
that isomorphism classes of $G$-equivariant bundle gerbe connections are classified by degree 3 differential equivariant cohomology.   The second theorem proves that a natural functor
\begin{equation}\label{Eq:FunctorIntro} G\GGrbc(M) \overset{\simeq}\longrightarrow \Grbc(\En G \times_G M) \end{equation}
is an equivalence of 2-groupoids.  As explained in works such as \cite{NikSch-Equiv}, a natural approach to equivariance in higher geometry is to use $\E G \times_G M$ and the language of sheaves.  The essential point of this paper is that when connections are involved, one should consider $\En G \times_G M$ instead of $\E G \times_G M$.

As a sample application of Theorem \ref{Thm:IsoClasses}, we consider the case where $G$ is a compact semisimple Lie group acting on itself by conjugation.  Theorem \ref{Thm:AdjointGerbe} gives a simple proof of the existence and uniqueness of certain $G$-equivariant gerbe connections on $G$.  Such gerbes have been considered in several other works, including \cite{BXZ-EquivGrbs, Brylinski-Grb00, MR1945806, Mein-Grb, MR3539489}.

The paper is structured in the following way.  Section \ref{Sec:Conventions} contains a brief exposition of existing knowledge that will be utilized in the paper.  The inclusion of this section serves the twofold purpose of establishing notations and conventions, and making the paper as self-contained as possible.  Section \ref{Sec:Defn} contains an explicit definition of the 2-groupoid $G\GGrbc(M)$ and proves many basic properties, most of which follow easily from existing work in \cite{Waldorf-More, Mein-Grb, Sti10}.  Section \ref{Sec:Stack} introduces the second model of $\Grbc(\En G\times_G M)$, and we use the Weil homomorphism to define the functor \eqref{Eq:FunctorIntro}.  Making use of a result by Tu--Xu \cite{TuXu15}, we are able to prove that \eqref{Eq:FunctorIntro} is a bijection on isomorphism classes, and we use the main result from \cite{Redden16b} to prove that it induces equivalences of automorphism groupoids.  Finally, Section \ref{Sec:Examples} contains a brief discussion of equivariant holonomy, and it proves a unique existence result for $\Ad$-equivariant gerbes on compact semisimple Lie groups. We also explain a natural model for equivariant connections on $p$-gerbes for $p>1$.

Finally, it is worth noting that gerbes, both with and without connection, have played an important role in mathematical and physical problems, including twisted/differential $K$-theory and Ramond-Ramond field classifications \cite{MR1911247, MR3195150, Park16a, 1512.07185}, Wess-Zumino-Witten models \cite{MR2318847}, string structures \cite{WaldorfStringCS}, and even topological insulators \cite{1512.01028, MR3681385}.  In each of these situations, it is natural to seek equivariant generalizations.  The tools developed in this paper should aid especially in answering existence/uniqueness questions and in better developing functorial constructions.  In particular, the formal properties of $\Grbc(\En G \times_G M)$, displayed in \eqref{Eq:Functoriality}, make equivariant gerbe connections excellent candidates for WZW terms in gauged sigma models  \cite{MR1151251, MR1311654, MR2770022,MR3079005, 1410.5500}.

\textbf{Acknowledgements.} B.P. thanks Hausdorff Research Institute for Mathematics and Max Planck Institute for Mathematics for their support and hospitality during his visits.  The authors thank Jim Stasheff for detailed comments on a previous draft.


\section{Background and conventions}\label{Sec:Conventions}

We begin by recalling background information and explaining notation, most of which follows the conventions from \cite{Redden16a, Redden16b, FH}; see \cite{BNV,MR2522659} for further details.  Throughout this paper, we work in the large category $\Man$ of smooth manifolds with smooth maps.  

\subsection{Sheaves of $\infty$-groupoids on the site of manifolds} \label{SEC.Sheaves.of.infty.groupoids}

While we use the language of sheaves and infinity categories, the reader unfamiliar with such terms should not worry.  Our sheaves are simply structures that naturally pullback, and our infinity-groupoids  arise from ordinary $n$-categorical structures with $n \leq 2$.  The particular model of $\Gpdinfty$ we use is the category of Kan complexes, viewed as a simplicially enriched full subcategory of simplicial sets.  Most of our arguments, though, are model-independent.

Any set is naturally a groupoid (a category whose morphisms are all invertible) with only the identity morphisms, and any groupoid determines a 2-groupoid with trivially defined 2-morphisms.  We freely use these natural embeddings (fully faithful functors)
\[ \Set \into \Gpd \into 2\text{-}\Gpd \into \Gpdinfty,\]
but we do not denote them by extra symbols.  For example, the map $\Gpd \into \Gpdinfty$ is the nerve construction, but we usually avoid the symbol $N$ and instead write $\Gpd \subset \Gpdinfty$.  While this may be technically imprecise, we believe it is easier to read and will cause no confusion.  

\begin{rem}Many of our constructions can be regarded as {\it strict} 2-categories.  However, the main results come after passing to $\Gpdinfty$, so issues such as strict 2-category versus bicategory will be irrelevant.  Therefore, we use the term 2-groupoid to refer to $\infty$-groupoids whose homotopy groups vanish above degree 2.  
\end{rem}

A morphism $\varphi \colon \cC \to \cC'$ in $\Gpdinfty$ is an {\it equivalence of $\infty$-groupoids} if it is a homotopy equivalence between the two underlying simplicial sets.  Since $\cC, \cC'$ are Kan complexes, $\varphi$ is an equivalence if and only if it induces isomorphisms between all homotopy groups.  When $\cC, \cC'$ are 2-groupoids, the homotopy groups $\pi_n(\cC^{(')})$ vanish for all $n>2$.  Hence $\varphi$ is an equivalence of 2-groupoids if and only if: $\varphi$ induces a bijection between isomorphism classes of objects $\pi_0(\cC) \iso \pi_0(\cC')$; and $\varphi$ induces an equivalence of categories $\Aut_{\cC}(x) \simeq \Aut_{\cC'}(\varphi(x))$ for all objects $x\in \cC$.  Note that since all morphisms are invertible, this implies the equivalence of categories  $\cC(x,y) \simeq \cC'(\varphi(x), \varphi(y))$ for all $x,y \in \cC$.  While we usually use the symbol $\simeq$ to denote equivalences, we sometimes use $\iso$ when an explicit functor has been constructed.

Define $\Pre_\infty$ as the $(\infty,1)$-category of $\Gpdinfty$-valued presheaves on the site of manifolds.  It is the collection of functors $\Pre_\infty \= \Fun(\Man^\op, \Gpdinfty)$.  We will be primarily interested in presheaves that satisfy a certain sheaf condition.

For $M \in \Man$, we call a surjective submersion $Y \xrightarrow{\pi} M$ a {\it cover} of $M$ and define the associated fiber products
\[ Y^{[k]} \= Y\times_M  \cdots \times_M Y = \{ (y_1, \ldots, y_k) \in Y^k \st \pi(y_1) = \cdots = \pi(y_k) \}, \]
with associated projection maps $\pi_{i_1 i_2 \cdots i_k} : Y^{[p]} \to Y^{[k]}$.  These manifolds form a simplicial manifold $Y^{[\bullet]} \colon \Delta^\op \to \Man$.  When $Y\to M$ is $G$-equivariant, we refer to it as a $G$-cover.

Define $\Shv_\infty$ as the full subcategory of $\Pre_\infty$ whose objects $\cF$ satisfy the following {\it sheaf/descent condition}: for any cover $U \to X$ with discrete fibers, the natural map 
\begin{equation}\label{Eq:SheafCond} \cF(X) \xrightarrow{\simeq} \holim_{\Delta} \Big[  \begin{tikzcd}[column sep=small]\cF(U) \ar[r, shift right=1.5] \ar[r, shift left=1.5]& \cF(U^{[2]}) \ar[l, dashed] \ar[r, shift right=3] \ar[r] \ar[r, shift left=3] & \cF(U^{[3]})  \ar[l, dashed, shift left=1.5] \ar[l, dashed, shift right=1.5]   \end{tikzcd} \cdots \Big]\end{equation}
is an equivalence in $\Gpdinfty$.  This recovers the usual sheaf condition when $\cF$ is $\Set$-valued, and it recovers the stack condition for presheaves of groupoids (for this reason objects in $\Shv_\infty$ are often called $\infty$-stacks).  The inclusion $\Shv_\infty \into \Pre_\infty$ has a left adjoint $\L$, which we call the {\it sheafification}, that gives natural equivalences of $\infty$-groupoids
\begin{equation}\label{Eq:L} \Pre_\infty(\cF, \cF') \iso \Shv_\infty(\L(\cF), \cF')  \end{equation}
for all $\cF \in \Pre_\infty$, $\, \cF' \in \Shv_\infty$.  Here, $\Pre_\infty(\cF, \cF') \in \Gpdinfty$ is the collection of maps $\cF \to \cF'$, which are natural transformations between the functors $\cF,\cF' \in \Fun(\Man^{\op}, \Gpdinfty)$.

\subsection{Relevant examples}
\begin{example}[Smooth maps and Yoneda] Any manifold $M\in \Man$ defines a sheaf $M \in \Shv_\infty$, whose value on a test manifold $X$ is given by the set of smooth functions
\[M(X) \= C^\infty(X, M) \in \Set \subset \Gpdinfty.\]
The Yoneda Lemma states there is a natural equivalence $\Pre_\infty(M, \cF) \iso \cF(M) \in \Gpdinfty$ for any $\cF \in \Pre_\infty$.
\end{example}

\begin{example}[Differential forms]Let $\O^n \in \Shv_\infty$ be the sheaf that assigns to any test manifold $X$ the differential $n$-forms $\O^n(X) \in \Set \subset \Gpdinfty$.  Note that $\O^n(X)$ is considered as a  category with no non-trivial morphisms, as opposed to the $\infty$-groupoid obtained from the de Rham complex $\cdots \to \O^{n-1}(X) \xrightarrow{d} \O^n(X)$.
\end{example}

\begin{example}[Bundles] For $X \in \Man$ and $K$ a Lie group, let $\Bunc[K](X)$ denote the groupoid of principal $K$-bundles with connection on $X$; morphisms are bundle isomorphisms preserving the connection and covering the identity map on $X$.  Principal bundles and connections naturally pullback along smooth maps, and they can be glued, thus defining a stack $\Bn K \in \Shv_\infty$ by 
\[ (\Bn K)(X) \= \Bunc[K](X) \in \Gpd \subset \Gpdinfty.\]
Similarly, we can ignore connections and define a stack by $(\B K)(X) \= \Bun[K](X)$.
\end{example}

For $G$ a Lie group, assumed to be compact throughout this paper, let $G\GMan$ denote the category whose objects are smooth manifolds equipped with a smooth (left) $G$-action, and whose morphisms are $G$-equivariant smooth maps.  We follow the convention that $M \in G\GMan$ has a left $G$-action, $P \in \Bun[G](X)$ has a right $G$-action, and inverses are used to switch between left/right actions when necessary.

\begin{example}[Quotient stacks]
If $M \in G\GMan$, define the {\it differential quotient stack} $\En G \times_G M \in \Shv_\infty$ by assigning to a test manifold $X$ the groupoid whose objects are $(P,\Theta,f)$, where $(P,\Theta) \in \Bunc[G](X)$ and $f\colon P \to M$ is a $G$-equivariant map.  Morphisms are given by morphisms in $\Bunc[G](X)$ compatible with maps $f$.  When $M=\pt$, which is trivially a $G$-manifold, then $\En G\times_G \pt = \Bn G$.  The {\it quotient stack} $\E G \times_G M \in \Shv_\infty$ is defined similarly, but without including connections.  The quotient stack serves as a replacement for the more familiar homotopy quotient $EG\times_G M$, and the adjective ``differential'' indicates the presence of local geometric data given by differential forms.
\end{example}

\begin{rem}\label{Rem:ActionGpd} For $M\in G\GMan$, the {\it action groupoid} is the Lie groupoid $ (G\times M \rightrightarrows M)$, which becomes the simplicial manifold $M\sslash G = G^\bullet \times M$ when viewed in $\Gpdinfty$.  The Yoneda embedding gives rise to a presheaf of groupoids with value 
\[ (M\sslash G)(X) \= N\left( C^\infty(X, G\times M) \rightrightarrows C^\infty(X, M)  \right) \in \Gpdinfty \]
on a test manifold $X$.  This presheaf does not satisfy the descent condition, but its sheafification is equivalent to the quotient stack
$ \L(M\sslash G) \xrightarrow{\simeq} \E G \times_G M \in \Shv_\infty.$
Similarly, the differential quotient stack $\En G \times_G M$ is equivalent to the sheafification of the differential action groupoid $M \nslash G \= N \left( G \times \O^1(\fg)\times M \rightrightarrows \O^1(\fg)\times M \right)$, where $\O^1(\fg)$ is the sheaf of $\fg$-valued 1-forms \cite{FH}. 
\end{rem}

When $BK$ is regarded as a fixed topological space, there is no natural way to choose a classifying map $X\to BK$ for a given bundle with connection $(P,\Theta) \in \Bunc[K](X)$.  In contrast, the Yoneda Lemma says that $(P,\Theta)$ is naturally equivalent to a map of stacks $X \xrightarrow{(P,\Theta)} \Bn K$.  Building upon this idea, we extend structures from manifolds to structures on sheaves in the following way.  For $\cF, \cM \in \Shv_\infty$,  {\it define} 
\[ \cF( \cM) \= \Shv_\infty(\cM, \cF) \in \Gpdinfty.\]
When $M$ is a manifold, we recover the usual definition of $\cF(M)$.  If $\cF(X) \in n\text{-}\Gpd$ for all $X$, then $\cF(\cM) \in n\text{-}\Gpd$ as well.

\begin{example}[Equivariant forms] There is a natural isomorphism $\O_G(M) \iso \O(\En G \times_G M)$ \cite{FH}, which we now explain.  The cochain complex $(\O^*_G(M), d_G)$ of {\it equivariant differential forms} is defined using the Weil model as the basic subcomplex
\begin{gather*} \O_G(M) \= \left( S\fg^* \otimes \Lambda \fg^* \otimes \O(M) \right)_{\bas}, \quad |S^1 \fg^*| = 2, \quad | \Lambda^1 \fg^* | = 1,
\end{gather*}
and for compact $G$ there is a natural de Rham isomorphism
\[ H^n(\O^*_G(M), d_G) \xrightarrow{\iso} H^n_G(M;\R) \= H^n(EG \times_G M;\R).\]
If $(P,\Theta) \in \Bunc[G](X)$ with curvature $\O$, the Weil homomorphism induces a natural map
\begin{equation}\label{Eq:WeilHomo} \begin{tikzcd}[row sep=0]
\Theta^*: \O^n_G(P) = \left(S\fg^* \otimes \Lambda \fg^* \otimes \O(P) \right)^n_{\bas} \ar[r] &\O^n(P)_\bas \iso \O^n(X) \\
\quad \quad \quad \quad \alpha \otimes \beta \otimes \gamma \ar[r,mapsto] & \alpha(\Omega) \wedge \beta(\Theta) \wedge \gamma.
\end{tikzcd} \end{equation}

By the Yoneda Lemma, a differential form $\omega \in \O^n(X)$ is equivalent to a map $X \xrightarrow{\omega} \O^n$.  For $M \in G\GMan$, it is therefore natural to define 
\begin{equation} \O^n (\En G \times_G M) \= \Shv_\infty( \En G \times_G M, \O^n) \in \Set \subset \Gpdinfty.
\end{equation}
Unpackaging this, a form $\omega \in \O^n(\En G \times_G M)$ assigns to every $X \xrightarrow{(P,\Theta,f)} \En G \times_G M$ a differential form $\omega(P,\Theta,f) \in \O^n(X)$; since $\O^n(X)$ has no non-trivial morphisms, a map of principal $G$-bundles $P' \xrightarrow{\varphi} P$ gives an equality of differential forms
\[ \begin{tikzcd}[column sep=small]
(P', \varphi^*\Theta) \ar[r, "\varphi"] \ar[d] & (P,\Theta) \ar[r,"f"] \ar[d] & M \\
X' \ar[r, "\bar{\varphi}"] & X
\end{tikzcd} \quad \mapsto \quad
 \omega(P', \varphi^*\Theta, f \o \varphi) = \bar{\varphi}^* \omega(P,\Theta, f) \in \O^n(X').
 \]
 
In \cite{FH}, Freed--Hopkins show the natural map induced by the Weil homomorphism gives an isomorphism, or bijection of sets, 
\begin{equation}\label{Eq:EquivForms} \O^n_G(M) \xrightarrow{\iso} \O^n(\En G \times_G M).\end{equation}
Explicitly, a form $\omega \in \O^n_G(M)$ evaluates on $(P,\Theta,f) \in (\En G \times_G M)(X)$  as the composition
\[  \O^n_G(M) \xrightarrow{f^*} \O^n_G(P) \xrightarrow{\Theta^*} \O^n(X),\]
with $\omega(P,\Theta, f) \= \Theta^* f^* \omega$.  The familiar equivariant differential $d_G$ is induced by the universal map between sheaves $\O^n \xrightarrow{d} \O^{n+1}$.  See \cite{GuilleminSternberg99} or \cite[Appendix A]{Redden16a} for further details of equivariant de Rham theory.
\end{example}

\begin{rem}One can equivalently use the Cartan model $\left(  S\fg^* \otimes \O(M) \right)^G$ for $\O_G(M)$, since Mathai--Quillen give a natural isomorphism of cochain complexes (not merely a quasi-isomorphism) between the Cartan and Weil models \cite{MathaiQuillen, MR1218928}.
\end{rem}

For $M \in G\GMan$, let $G\GBunc[K](M)$ denote the groupoid of {\it $G$-equivariant principal $K$-bundles on $M$ with $G$-invariant connection}; morphisms are $G$-equivariant bundle isomorphisms that preserve the connections.  For brevity, we refer to $(Q,\Theta) \in G\GBunc[K](M)$ as a $G$-equivariant $K$-connection on $M$.  Equivariant connections always exist when $G$ is compact, meaning the forgetful map $G\GBunc[K](M) \to G\GBun[K](M)$ is surjective on objects.

Suppose that $P \in G\GMan$ has a free $G$-action, so that we can regard $P \in \Bun[G](X)$.  Any $Q \in G\GBun[K](P)$ naturally descends to a bundle $Q/G \in \Bun[K](X)$.  However, a $K$-connection $\Theta_Q$ on $Q$ will descend to $Q/G$ if and only if it is $G$-basic.  If $(P,\Theta_P) \in \Bunc[G](X)$, then there is a canonical way to modify the connection on $Q$ by $\Theta_P$, producing a $G$-basic connection 
\[ (1- \iota_{\Theta_P})\Theta_Q \in \O^1(Q, \fk)_{G\text{-}\bas} \iso \O^1(Q/G;\fk),\]
as explained in \cite[Section 4.2]{Redden16a} and \cite[Section 3]{Redden16b}.  This connection form on $Q/G$ equals the image of $\Theta_G \in \O_G^1(Q;\fk)$, the $G$-equivariant extension of $\Theta_Q$, under the Weil homomorphism \eqref{Eq:WeilHomo}.  Reusing the same symbol, for any
$(P,\Theta_P) \in \Bunc[G](X)$ we denote the natural descent functor by 
\begin{equation}\label{Eq:Theta*}  G\GBunc[K](P) \xrightarrow{\Theta_P^*} \Bunc[K](X). \end{equation}

\begin{example}[Equivariant connections]\label{Ex:EquivConns}
The construction \eqref{Eq:Theta*} naturally induces a functor $G\GBunc[K](M) \to \Shv_\infty(\En G \times_G M, \Bn K)$, as explained in both \cite[Section 4.2]{Redden16a} and \cite{Fre16}.  For $(Q,\Theta_Q) \in G\GBunc[K](M)$, define $(\En G \times_G M)(X) \to (\Bn K)(X)$ for every test manifold $X$ by
\begin{equation}\label{Eq:ConnsQuot} \begin{tikzcd}[row sep=15] (P,\Theta_P) \ar[r, "f"] \ar[d] & M \\ X \end{tikzcd} \quad \longmapsto \quad  \Theta_P^* (f^*(Q,\Theta_Q)) \in \Bunc[K](X),\end{equation}
using the composition
\[ G\GBunc[K](M) \xrightarrow{f^*} G\GBunc[K](P) \xrightarrow{\Theta_P^*} \Bunc[K](X).\]
This construction establishes a functor between groupoids
\begin{equation}\label{Eq:EquivConns} \begin{tikzcd} G\GBunc[K](M) \ar[r, "\iso", shift left] & \Bunc[K](\En G \times_G M) \= \Shv_\infty(\En G\times_G M, \Bn K),  \ar[l, shift left] \end{tikzcd} \end{equation}
which was proven to be an equivalence in \cite{Redden16b} through the construction of an explicit inverse functor.  Furthermore, consider the map $\En G\times_G M \to \Bn K$ induced by a specific $(Q,\Theta) \in G\GBunc[K](M)$.  As shown in \cite[Theorem 5.3]{Redden16a} and \cite[Theorem 8]{Fre16}, the induced map on differential forms 
\[ (S^n \fk^*)^K \iso \O^{2n}(\Bn K) \longrightarrow \O^{2n}(\En G\times_G M) \iso \O^{2n}_G(M) \]
recovers the familiar equivariant Chern--Weil homomorphism of Berline--Vergne \cite{BV83}.
\end{example}

\begin{example}[Gerbes] When working with gerbes, we always consider gerbes with band $S^1$.  Let $\Grbc(X) \in 2\text{-}\Gpd$ be the 2-groupoid of bundle gerbes with connection on $X$, as defined in \cite{Waldorf-More} and described in Section \ref{Sec:Defn}.  As explained in \cite[Section 4.1]{NikSch-Equiv}, bundle gerbes satisfy the descent condition \eqref{Eq:SheafCond}, and we denote the corresponding 2-stack by $\BBn S^1 \in \Shv_\infty$, where
\[ (\BBn S^1)(X) \= \Grbc(X) \in 2\text{-}\Gpd \subset \Gpdinfty.\]
When $M \in G\GMan$, the discussion of equivariant connections in Example \ref{Ex:EquivConns} suggests that a ``correct'' definition of $G$-equivariant gerbe connections on $M$ is given by $\Grbc(\En G \times_M G) \= \Shv_\infty( \En G \times_G M, \BBn S^1)$.  This claim is what we ultimately prove in Theorems \ref{Thm:IsoClasses} and \ref{Thm:Equiv2Gpd}.
\end{example}

\subsection{Cohomology and differential cohomology}\label{Subsec:Coh}
Using sheaves of $\infty$-groupoids on manifolds allows us to easily work with cohomology and differential cohomology.  In doing so, we often use the Dold--Kan equivalence between non-negatively graded chain complexes of abelian groups and simplicial abelian groups 
\[ \DK \colon \Ch \overset{\simeq}\longrightarrow \Ab^{\Delta^{\op}} . \]
Let $A$ be a sheaf of abelian Lie groups, viewed as a $\Set$-valued sheaf; i.e. $A(X) \in \Ab \subset \Set \subset \Gpdinfty$ for every $X \in \Man$.  By placing $A$ in degree $n$ of a chain complex, using the Dold--Kan functor and sheafifying, one gets an object 
\[\L (\DK(A[-n])) = \L(\DK( \cdots \to  0 \to A \to 0 \to \cdots \to 0)) \in \Shv_\infty.\]
For $\cM \in \Shv_\infty$, define the sheaf cohomology with coefficients in $A$ by 
\begin{equation}\label{Eq:SheafCohDef} H^n_{\Shv} (\cM; A) \= \pi_0 \Shv_\infty \big( \cM, \L(\DK(A[-n])) \big) \in \Ab. \end{equation}
If $M \in \Man$, then $H^*_{\Shv}(M;A)$ recovers the traditional sheaf cohomology groups \cite{MR0341469, MR2522659}. 

When $A=A^\delta$ is a discrete group, then $\K(A,n) \= \L(\DK(A[-n]))$ is a $\Shv_\infty$ analogue of an Eilenberg--MacLane space.   Since $\K(A,n)$ is homotopy-invariant, the arguments and techniques from \cite{BNV} imply natural isomorphisms
\begin{equation}\label{Eq:EquivCoh} \begin{split} H^n_{\Shv}(\En G \times_G M; A) &\iso H^n_{\Shv}(\E G \times_G M; A) \\ &\iso \pi_0\Top(EG \times_G M, K(A,n)) \iso H^n_G(M;A), \end{split} \end{equation}
as proven in \cite[Proposition 6.7]{Redden16a}.

The differential cohomology groups, also known as Cheeger--Simons differential characters \cite{CS85} or Deligne cohomology \cite{MR0498551, Bry93}, can be defined in a similar way \cite{BNV, MR3019405, MR3335251}.  Following the notation from \cite{Redden16a}, let $\Kh(\Z, n) \in \Shv_\infty$ be a homotopy pullback in the diagram 
\[ \begin{tikzcd} \Kh(\Z, n) \ar[r]\ar[d]& \K(\Z,n) \ar[d]  \\
\O^n_{\operatorname{cl}} \ar[r] & \K(\R, n). 
\end{tikzcd}\]
For $M\in G\GMan$ we define the {\it differential equivariant cohomology} groups as 
\[ \HG^{n}(M) \= \HG[]^n(\En G \times_G M) \= \pi_0 \Shv_\infty( \En G \times_G M, \Kh(\Z, n)).\]
Just as with traditional differential cohomology, the groups $\HG^*(M)$ sit in the following three short exact sequences \cite{Redden16a}.
\begin{align}
\label{ses1}  0 \longrightarrow H^{n-1}_G(M; \R/\Z) \longrightarrow &\HG^n(M) \longrightarrow \O^n_G(M)_\Z \longrightarrow 0 \\
\label{ses2} 0 \longrightarrow \frac{\O^{n-1}_G(M)}{\O^{n-1}_G(M)_\Z} \longrightarrow &\HG^n(M) \longrightarrow  H^n_G(M;\Z) \longrightarrow 0 \\
\label{ses3} 0 \longrightarrow \frac{H^{n-1}_G(M;\R)}{H^{n-1}_G(M;\R)_\Z} \longrightarrow &\HG^n(M) \longrightarrow A^n_G(M) \longrightarrow 0
\end{align}
Here $A^n(M)$ denotes the pullback (of sets) in the diagram
\begin{equation}\label{Eq:PullbackA} \begin{tikzcd} \HG^n(M) \ar[r, two heads] \ar[d, two heads] & H^n_G(M;\Z) \ar[d, two heads] \\
\O_G^n(M)_\Z \ar[r, two heads] & H^n_G(M;\R)_\Z 
\end{tikzcd} \end{equation}
in which all maps are surjective.  The subscript $\Z$ denotes the closed forms with integer periods and their image in de Rham cohomology.  The groups $\HG^n(M)$ were also defined, independently and using different techniques, in \cite{KubelThom-EDC}.

The equivalence $\BBn S^1 \simeq \L ( \DK (S^1 \xrightarrow{d} \O^1 \xrightarrow{d} \O^2 )) \simeq \Kh(\Z,3)$, which is sometimes used as the definition of $\BBn S^1$, immediately implies
\begin{equation}\label{Eq:DiffH3} \pi_0  \Shv_\infty (\En G \times_G M, \BBn S^1) \iso \HG^3(M).
\end{equation}
This gives further evidence that $\Grbc (\En G \times_G M)$ is a desirable model for equivariant gerbe connections.  


\section{The 2-groupoid of equivariant bundle gerbe connections}\label{Sec:Defn}

For a $G$-manifold $M \in G\GMan$, we consider a geometrically defined 2-category $G\GGrbc(M)$ of {\it $G$-equivariant bundle gerbe connections}, which is an abbreviated way of saying ``$G$-equivariant bundle gerbes on $M$ with $G$-equivariant connection.''   Essentially, this is obtained from Waldorf's 2-groupoid $\Grbc(M)$ of bundle gerbes with connection (and invertible 1-morphisms) \cite{Waldorf-More} and appending the adjective ``$G$-equivariant'' to everything; $G\GGrbc(M)$ reduces to $\Grbc(M)$ when $G$ is the trivial group.  The objects in $G\GGrbc(M)$ were originally defined by Meinrenken \cite{Mein-Grb} and further developed by Sti\'enon \cite{Sti10} and Tu--Xu \cite{TuXu15}.  The key point is that the curving 2-form is an {\it equivariant} 2-form in $\O^2_G(Y)$, as opposed to other possible definitions  \cite{MR2770022, Gomi-EquivGrbs}.  We define the groupoid of morphisms in $G\GGrbc(M)$ by adapting Waldorf's structures \cite{Waldorf-More} to the equivariant setting.

For $M \in G\GMan$,  define a {\it $G$-cover} to be a $G$-equivariant surjective submersion $Y \xrightarrow{\pi} M$; the fiber products $Y^{[p]}$ and projections $\pi_{i_1 \cdots i_k}$ are naturally $G$-equivariant.  Let $\delta$ denote the simplicial differential on $\O^k(Y^{[\bullet]})$, which is given by the alternating sum of pullbacks.  When $G$ is compact, the natural sequences 
\begin{gather}\label{Eq:CoverSeq1} 0 \to \O^k(M)  \xrightarrow{\pi^*} \O^k(Y) \xrightarrow{\delta} \O^k(Y^{[2]}) \xrightarrow{\delta} \O^k(Y^{[3]}) \xrightarrow{\delta} \cdots \\
\label{Eq:CoverSeq2} 0 \to \O^k(M)^G  \xrightarrow{\pi^*} \O^k(Y)^G \xrightarrow{\delta} \O^k(Y^{[2]})^G \xrightarrow{\delta} \O^k(Y^{[3]})^G \xrightarrow{\delta} \cdots \\
\label{Eq:CoverSeq3} 0 \to \O^k_G(M)  \xrightarrow{\pi^*} \O^k_G(Y) \xrightarrow{\delta} \O^k_G(Y^{[2]}) \xrightarrow{\delta} \O^k_G(Y^{[3]}) \xrightarrow{\delta} \cdots \end{gather}
are exact, as shown in \cite[Lemma 3.3]{Sti10}.  

We also use $\delta$ to denote the analogous construction for $S^1$-bundles.  In particular, for $L \in G\GBun(Y)$, 
\[ \delta L = \pi_2^*L \otimes \pi_1^*L^{-1} \in G\GBun(Y^{[2]}),\]
and for $L \in G\GBun(Y^{[2]})$, then
\[ \delta L = \pi_{23}^* L \otimes \pi_{13}^*L^{-1}\otimes \pi_{12}^* L \in G\GBun(Y^{[3]}).\]
The construction $\delta$ also extends to $G$-equivariant principal $S^1$-bundles with $G$-invariant connection on $Y^{[k]}$, which we denote by $G\GBunc(Y^{[k]})$.  We denote the trivial bundle (with trivial connection) by $1$, as it is naturally a unit under the tensor product.  

\begin{rem}Our notation implicitly uses the natural equivalence between principal $S^1$-bundles with connection and Hermitian line bundles with connection.  For example, the product of two principal $S^1$-bundles is actually defined by the associated bundle construction, though we denote it by the tensor product.  While our notation is technically imprecise, it is easier to read and should cause no confusion.  
\end{rem}

We now write the full definition of $G\GGrbc(M)$ for completeness, but almost all of Waldorf's arguments immediately generalize.  This is because $G$-equivariant submersions and $G$-equivariant bundles pull back along $G$-equivariant maps, and the tensor product of two $G$-equivariant line bundles is again a $G$-equivariant line bundle.  We will mainly reference Waldorf's proofs and only provide further details when some care is required. 

\subsection{The definition}\label{subsec:Defn}
\begin{defn}[Objects, c.f. \cite{Waldorf-More, Mein-Grb, Sti10}]An object $\cLh = (Y,L,\nabla, B, \mu) \in G\GGrbc(M)$ consists of:
\begin{itemize}
\item A $G$-cover $Y \xrightarrow{\pi} M$,
\item $(L,\nabla) \in G\GBunc[S^1](Y^{[2]})$,
\item $B \in \O^2_G(Y)$ satisfying $\curv_G(\nabla) = \pi_2^*B-\pi_1^*B \in \O^2_G(Y^{[2]})$,
\item and an isomorphism $\mu \colon \pi_{12}^*(L,\nabla) \otimes \pi_{23}^* (L,\nabla) \to \pi_{13}^*(L,\nabla)$  in $G\GBunc[S^1](Y^{[3]})$,  called the {\it bundle gerbe multiplication}, that is associative in that the following diagram in $G\GBunc(Y^{[4]})$ commutes.
\[ \begin{tikzcd}[column sep=huge] \pi_{12}^*L \otimes \pi_{23}^*L \otimes \pi^*_{34}L \ar[r, "\pi_{123}^*\mu \otimes \Id"] \ar[d, "\Id \otimes \pi^*_{234} \mu"'] & \pi^*_{13}L \otimes \pi^*_{34} L  \ar[d, "\pi^*_{134} \mu"] \\
\pi^*_{12} L \otimes \pi^*_{24} L \ar[r, "\pi^*_{124}\mu"'] & \pi^*_{14} L 
\end{tikzcd} \]
\end{itemize}
\end{defn}

\begin{rem}Readers feeling overwhelmed by the abundance of pullbacks can suppress them by writing $\pi_{ij}^*(L,\nabla)$ as $(L_{ij},\nabla^{ij})$.  The key properties are then written as $\curv_{G}(\nabla^{ij})= B_j-B_i$,  $\mu: L_{i j} \otimes L_{jk} \to L_{ik}$, and $\mu(\mu(L_{ij}\otimes L_{jk}) \otimes L_{kl}) = \mu(L_{ij}\otimes \mu(L_{jk} \otimes L_{kl}))$.
\end{rem}

\begin{rem}\label{Rem:EquivalentDefn}The isomorphism $\mu$ is equivalent to $\upmu \colon \delta(L,\nabla) \to 1$  satisfying $\delta \upmu \iso \Id$ under the natural isomorphism $\delta(\delta(L,\nabla)) \iso 1$.
\end{rem}

\begin{defn}[1-morphisms, c.f. \cite{Waldorf-More}]\label{Defn:1morph} For $\cLh_i = (Y_i, L_i, \nabla_i, B_i, \mu_i) \in G\GGrbc(M)$, an isomorphism $\cLh_1 \xrightarrow{\cKh} \cLh_2$, denoted $\cKh =(Z,K,\nabla_K,\alpha)\in G\GGrbc(\cLh_1, \cLh_2)$, consists of:
\begin{itemize}
\item a $G$-cover $\zeta\colon Z \to Y_1 \times_M Y_2$,
\item $(K,\nabla_K) \in G\GBunc(Z)$ such that $\curv_G(\nabla_K) = \zeta^*(B_2 - B_1) \in \O^2_G(Z)$,
\item and an isomorphism $\alpha\colon (L_1,\nabla_1) \otimes \zeta_2^*(K,\nabla_K) \to \zeta_1^*(K,\nabla_K) \otimes (L_2, \nabla_2)$ in $G\GBunc(Z\times_M Z)$ that is compatible with $\mu_1$ and $\mu_2$, in that the following diagram in $G\GBunc(Z\times_M Z \times_M Z)$ commutes.
\[ \begin{tikzcd}[column sep=huge] \zeta^*_{12} L_1 \otimes \zeta^*_{23} L_1 \otimes \zeta_3^* K \ar[r, "\mu_1 \otimes \Id"] \ar[d, "\Id \otimes \zeta^*_{23} \alpha" '] & \zeta^*_{13} L_1 \otimes \zeta_3^* K  \ar[dd, "\zeta^*_{13}\alpha"] \\
\zeta^*_{12}L_1 \otimes \zeta^*_2 K \otimes \zeta^*_{23} L_2  \ar[d, "\zeta^*_{12}\alpha \otimes \Id" '] \\
\zeta^*_1 K \otimes \zeta^*_{12} L_2 \otimes \zeta^*_{23} L_2 \ar[r, "\Id \otimes \mu_2"'] & \zeta^*_1 K \otimes \zeta^*_{13} L_2
\end{tikzcd} \]
\end{itemize}
\end{defn}

\begin{rem}Throughout the paper, we avoid writing pullbacks when they are obvious so as not to further complicate notation.  For example, in the above definition we use $L_i \in G\GBunc(Z\times_M Z)$ to denote the bundle obtained by pulling back $L_i \in G\GBunc(Y_i^{[2]})$ along the obvious map $Z\times_M Z \to Y_i^{[2]}$.
\end{rem}

\begin{rem}The 1-morphisms in \cite{Waldorf-More} are more general than those in Definition \ref{Defn:1morph}.  We restrict restrict ourselves to Waldorf's invertible morphisms, which are given by vector bundles of rank one.  An analog of Waldorf's ``more morphisms'' category can be obtained by letting $(K,\nabla) \in G\GBunc[U(n)](M)$, but we do not consider them in this paper.
\end{rem}

\begin{defn}[2-morphisms, c.f. \cite{Waldorf-More}]For $\cKh_1, \cKh_2 \in G\GGrbc(\cLh_1, \cLh_2)$, a 2-morphism $\cJh \colon \cKh_1 \Rightarrow \cKh_2$ in $G\GGrbc(M)$, which can also be viewed as a 1-morphism in the groupoid $G\GGrbc(\cLh_1, \cLh_2)$, is denoted $\cJh = (W, \beta)$ and is an equivalence class of the following:
\begin{itemize}
\item a  $G$-equivariant surjective submersion $\omega \colon W \to Z_1 \times_{(Y_1 \times_M Y_2)} Z_2$, 
\item  and an isomorphism $\beta \colon (K_1,\nabla_1) \to (K_2, \nabla_2)$ in $G\GBunc(W)$ compatible with $\alpha_1$ and $\alpha_2$, in the sense that the following diagram in $G\GBunc(W\times_M W)$ commutes.
\[ \begin{tikzcd} L_1 \otimes \omega_2^*A_1 \ar[r, "\alpha_1"] \ar[d, "\Id \otimes \omega_2^* \beta"'] & \omega_1^* A_1 \otimes L_2 \ar[d, "\omega_1^* \beta \otimes \Id"] \\ 
L_1 \otimes \omega_2^* A_2 \ar[r, "\alpha_2"'] & \omega_1^* A_2 \otimes L_2
\end{tikzcd} \]
\end{itemize}
Two $2$-morphisms $\cJh_i=(W_i, \beta_i) \colon \cKh \Rightarrow \cKh'$ are equivalent if there exists $G$-covers $\pi_i:X\ra W_i$ for $i=1,2$ such that $\om_1 \circ \pi_1=\om_2 \circ \pi_2$ and $\pi_1^*\be_{1}=\pi_2^*\be_{2}$.
\end{defn}

Morphisms can be easily composed using pullbacks, since the definitions of morphisms include the additional datum of a $G$-cover.  These details are carefully discussed in \cite[Section 1]{Waldorf-More}, and they immediately generalize to the equivariant setting.  Composition in the groupoid $G\GGrbc(\cLh_1, \cLh_2)$, also referred to as {\it vertical composition} and denoted by $\bullet$, is defined by the following.  If $\cJh \colon \cKh_1 \Rightarrow \cKh_2$ and $\cJh' \colon \cKh_2 \Rightarrow \cKh_3$, then $\cJh' \bullet \cJh \colon \cKh_1 \Rightarrow \cKh_3$ is given by the $G$-cover $\wt{W} = W\times_{Z_2} W' \to Z_1 \times_{Y_1\times_M Y_2} Z_3$ and isomorphism $\beta' \o \beta \colon K_1 \to K_3$ in $G\GBunc(\wt{W})$.  The {\it horizontal composition} of 1-morphisms in $G\GGrbc(M)$ is defined in a similar way.  For $\cKh \in G\GGrbc(\cLh_1, \cLh_2)$, $\cKh' \in G\GGrbc(\cLh_2, \cLh_3)$, the composition $\cKh' \o \cKh = (\wt{Z}, \wt{K}, \wt{\nabla}, \wt{\alpha}) \in G\GGrbc(\cLh_1, \cLh_3)$ is defined by 
$\wt{Z} \= Z\times_{Y_2}  Z' \to Y_1\times_M Y_3$, with $(\wt{K},\wt{\nabla}) = (K,\nabla_K) \otimes(K',\nabla_{K'})$ and with isomorphism 
$\wt{\alpha}=( \Id_{\zeta_1^{*} K} \otimes \alpha' ) \o ( \alpha \otimes \Id_{\zeta_2^{'*}K'} )$ in $G\GBunc( \wt{Z}\times_M \wt{Z})$.  Horizontal composition of 2-morphisms is defined as well, and we refer the interested reader to \cite[pp. 250-1]{Waldorf-More} for full details.

By ignoring all the connections and differential forms, we may similarly define $G\GGrb(M)$, the 2-groupoid of $G$-equivariant gerbes without connection.  There is a natural forgetful functor
\begin{align*}
G\GGrbc(M) &\longrightarrow G\GGrb(M) \\
\cLh=(Y,L,\nabla, B, \mu) &\longmapsto \cL = (Y, L, \mu),
\end{align*}
which we denote by eliminating the hat $\text{ }\widehat{ }\text{ }$ symbol.  By following the same arguments from \cite[Section 1]{Waldorf-More}, we conclude the following.

\begin{thm}[c.f. \cite{Waldorf-More}]If $M \in G\GMan$, then $G\GGrbc(M)$ and $G\GGrb(M)$ are well-defined 2-groupoids.
\end{thm}

\subsection{Basic properties}

It is evident that a $G$-equivariant map $M_1 \xrightarrow{f} M_2$ induces  functors
\begin{align}\label{Eq:Pullback}
\begin{split}G\GGrbc(M_2) &\xrightarrow{f^*} G\GGrbc(M_1), \\
G\GGrb(M_2) &\xrightarrow{f^*} G\GGrb(M_1), \end{split}
\end{align}
given by pulling back all relevant $G$-equivariant structures along $G$-equivariant maps. 

The exact sequence \eqref{Eq:CoverSeq2}, which assumes $G$ is compact, implies that every $\cL \in G\GGrb(M)$ admits a $G$-equivariant connection \cite[Propostion 3.2]{Sti10}.  In other words, the forgetful functor $G\GGrbc(M) \to G\GGrb(M)$ is surjective on objects.  Furthermore, the exact sequence \eqref{Eq:CoverSeq3} implies that any $\cLh \in G\GGrbc(M)$ has a unique equivariant curvature $\curv_G(\cLh) \in \O^3_G(M)$ defined by the property
\begin{equation}\label{Eq:EquivCurv} \pi^* \curv_G (\cLh) = d_G B \in \O_G^3(Y), \end{equation}
and that $d_G \curv_G(\cLh) = 0 \in \O_G^4(M)$ \cite[Lemma 3.3]{Sti10}.  As one would expect, the equivariant curvature depends only on the isomorphism class of $\cLh$.

\begin{prop}\label{Prop:Curvature}If $\cLh_1$, $\cLh_2 \in G\GGrbc(M)$ are isomorphic, then 
\[\curv_G(\cLh_1) = \curv_G(\cLh_2) \in \O^3_G(M).\]
\end{prop}
\begin{proof}Assume there is an isomorphism $\cKh \colon \cLh_1 \to \cLh_2$, as described in Definition \ref{Defn:1morph}.  The $G$-equivariant surjective submersions $Z \to Y_i \to M$ induce injective maps $\O_G(M) \to \O_G(Y_i) \to \O_G(Z)$, which we suppress in the following calculation:
\[ \curv_G(\cLh_2) - \curv_G(\cLh_1) = d_G(B_2-B_1) = d_G \curv_G(\nabla^K) = 0 \in \O^3_G(Z).\]
The first equality comes from the definition of equivariant gerbe curvature in \eqref{Eq:EquivCurv}. The second equality is in Definition \ref{Defn:1morph} of 1-morphisms.  The last equality follows because the equariant curvature of an equivariant $S^1$-bundle connection is equivariantly closed.
\end{proof}

There is a canonical trivial gerbe $\cI \in G\GGrb(M)$ given by $Y = M$, $L=1$, and $\mu \iso \Id$ under the canonical identification $1\otimes 1 \iso 1$. Using the same $G$-cover $\Id\colon M\to M$, any $B \in \O^2_G(M)$ satisfies $\delta B = 0 = \curv_G(1)$, where $1 \in G\GBunc(M)$ denotes the trivial bundle with trivial connection.  Hence, any $B \in \O^2_G(M)$ defines $\cIh_B \= (M, 1, B, \Id) \in G\GGrbc(M)$ with curvature $\curv_G(\cIh_B) = d_G B \in \O^3_G(M)$ and with underlying topological gerbe $\cI$.  The trivial gerbe with trivial connection is then denoted $\cIh_0 \in G\GGrbc(M)$.

The natural Picard structure on the category of complex lines, given by the tensor product and the dual, naturally makes the category of $S^1$-bundles into a Picard groupoid and the category of gerbes into a Picard 2-groupoid.  These structures naturally carry over to the equivariant setting, and the same proofs that Waldorf gives in \cite{Waldorf-More} show that $G\GGrbc(M)$ is a symmetric monoidal 2-category in which all 1- and 2-morphisms are invertible, and all objects are invertible under the tensor product (symmetric monoidal) structure.  The curvature 3-forms are additive with respect to tensor product,
\[ \curv_G(\cLh_1\otimes \cLh_2) = \curv_G(\cLh_1) + \curv_G(\cLh_2).\]
The tensor units are $\cIh_0 \in G\GGrbc(M)$ and $\cI \in G\GGrb(M)$.  In particular, there is a natural adjunction
\begin{equation}\label{Eq:Adjunction} G\GGrbc(\cLh_1, \cLh_2) \xrightarrow{\iso} G\GGrbc(\cIh_0, \cLh_1^{-1} \otimes \cLh_2) \end{equation}
and a canonical trivialization $\cIh_0 \xrightarrow{\iso} \cLh^{-1} \otimes \cLh$, which gives the natural equivalence of groupoids 
\begin{equation}\label{Eq:Automorphisms} \Aut(\cLh) = G\GGrbc(\cLh, \cLh) \iso G\GGrbc(\cIh_0, \cLh^{-1} \otimes \cLh) \iso  G\GGrbc(\cIh_0, \cIh_0) = \Aut(\cIh_0).\end{equation}


\subsection{Fiber product morphisms}\label{Subsec:GrbProps}

For $\widehat{\cL}_1$, $\widehat{\cL}_2\in G\GGrbc(M)$, we consider the subgroupoid of {\it fiber product morphisms} $G\GGrbc^{\FP}(\cLh_1, \cLh_2) \subset G\GGrbc(\cLh_1, \cLh_2)$.  Objects are 1-morphisms $\cKh = (Z,K,\nabla_K,\alpha)$ whose equivariant surjective submersion $\zeta \colon Z \to Y_1 \times_M Y_2$ is the identity; similarly, morphisms are $\cJh = (W,\beta)$ with $W\to Y_1\times_M Y_2$ the identity.  Ignoring differential data, the same conditions also define $G\GGrb^{\FP}(\cL_1,\cL_2) \subset G\GGrb(\cL_1,\cL_2)$.  These fiber product morphisms, at least in the non-equivariant case, are often called stable isomorphisms \cite{StevensonThesis}.

In \cite[Theorem 2.1]{Waldorf-More} Waldorf shows that the inclusion $\Grbc^{\FP}(\cLh_1, \cLh_2) \into \Grbc(\cLh_1,\cLh_2)$ is an equivalence of categories.  The benefit is immediate: allowing arbitrary covers is a more natural definition of isomorphisms, but fiber product morphisms are easier to work with in practice.  This equivalence of categories extends to the $G$-equivariant case, as stated in Proposition \ref{Prop:StableIso} below, and the proof is the same as Waldorf's argument in \cite{Waldorf-More}.  We include some of these details, however, since the reader could potentially mistranslate one of Waldorf's arguments to the equivariant setting, as explained below.

Let $Z \xrightarrow{\zeta} M$ be a $G$-cover. Define the category $G\GBunc^{\desc}(\zeta)$ as follows. Objects are triples $(K,\nabla,\mu)$, where $(K,\nabla) \in G\GBunc(Z)$ and $\mu:\zeta_1^*(K,\nabla) \ra \zeta^*_2(K,\nabla)$ is an isomorphism over $Z^{[2]}$ satisfying the cocycle identity $\zeta^*_{13}\mu = \zeta^*_{23}\mu \o \zeta^*_{12} \mu$ on $Z^{[3]}$. Morphisms from $(K,\nabla,\mu)$ to $(K',\nabla',\mu')$ are isomorphisms $\al:(K,\nabla)\ra (K',\nabla')$ such that the following diagram commutes. \beq\label{EQN.Decent.Category.Morphism.Condition} \xymatrix{
\zeta_1^*(K,\nabla) \ar[r]^{\zeta_1^*\al} \ar[d]^{\mu} & \zeta_1^*(K',\nabla') \ar[d]^{\mu'}\\
\zeta_2^*(K,\nabla) \ar[r]^{\zeta_2^*\al}  & \zeta_2^*(K',\nabla')
}\eeq

The following Lemma \ref{Lem:Descent} is well-known in the non-equivariant case of principal bundles considered without connections; see Brylinski \cite[p.187]{Bry93}.  The version in  \cite{Waldorf-More} is not proven directly, but instead Waldorf utilizes the well-known fact that $S^1$-bundles with connection form a stack on the site of manifolds, something that $G\GBunc(-)$ is not.  The descent bundle, however, can be defined as a limit, and it will be $G$-equivariant if the original diagram was in $G\GMan$.  Since we could not find a reference for the equivariant version of descent, we include a full proof.

\begin{lemma}[Descent for equivariant bundles]\label{Lem:Descent} Let $Z \xrightarrow{\zeta} M$ be a $G$-cover. The functor $$\zeta^*:G\GBunc(M) \ra G\GBunc^{\desc}(\zeta)$$ is an equivalence of groupoids.
\end{lemma}
\begin{proof} We shall prove that the functor $\zeta^*$ is essentially surjective and fully faithful. We verify essential surjectivity by taking an object in the descent groupoid and defining a $G$-equivariant $S^1$-bundle with connection over $M$ through a quotient construction using the descent isomorphism. For any $(K,\te, \mu)\in G\GBunc^{\desc}(\zeta)$ we define a relation $\sim$ on $K$ as follows: For $k\in K_{z_1}$ and $k'\in K_{z_2}$,  $k\sim k'$ if $\mu((z_1,z_2),k)=((z_1,z_2),k')$. Since $\mu$ is an isomorphism of $S^1$-bundles satisfying the cocycle identity, the relation $\sim$ is an equivalence relation. We construct a $G$-equivariant $S^1$-bundle $S$ over $M$ as follows. Define $S:=K/\sim$ and topologize it by the map $K\ra K/\sim$. The projection map $\pi:S\ra M$ is defined by $[k]\mapsto \zeta\circ\pi_K(k)$ for any $[k]\in S$, which is clearly well-defined. Here $\pi_K$ is the projection map $K\ra Z$. For any $x\in M$, $S^1$ acts on $S_x:=\pi^{-1}(x)$ by $[k]\cdot r:=[k\cdot r]$ for all $[k]\in S_x$ and $r\in S^1$, and the action is free and transitive, since $S^1$ acts on $K$ as such. To show $S$ is locally trivial, we take any $x\in M$ and an open neighborhood $U$ of $x$ such that $\zeta$ splits locally by a smooth (non-equivariant) map $\si:U\ra Z$ and the image of $\si$ is properly contained in an open neighborhood $V$ of $\si(x)$ over which $K$ is locally trivial by a local trivialization $\vph:K|_{V}\ra V\times S^1$. We define a map $S|_{U}\ra U\times S^1$ by $[k]\mapsto (\zeta\circ\pi_K(k),\vph|_{V\cap \si(U)}(k))$. This map is well-defined because, from the construction of $\si$, there is only one representative $k$ of $[k]$ that has $\pi_K(k)\in \si(U)$, and it is clearly $S^1$-equivariant. We now endow a $G$-action on $S$ by $g\cdot [k]:=[g\cdot k]$ making $S\in G\GBun(M)$ since $K$, $\mu$ and $\zeta$ are $G$-equivariant, and the action commutes with the right $S^1$-action on $S$.

Now we define a $G$-invariant connection $\Te$ on $S$. Let $v_1\in T_{k_1}K$ and $v_2\in T_{k_2}K$, where $\pi_K(k_i)=z_i$ for $i=1,2$. We define a relation $v_1\sim_* v_2$ if there is a path $(\al(t),\ga(t))$ in $\zeta_1^*K$ and a path $(\be(t),\de(t))$ in $\zeta_2^*K$ satisfying $\al(0)=(z_1,z_2)=\be(0)$, $\ga(0)=k_1$, $\ga'(0)=v_1$, $\de(0)=k_2$ and $\de'(0)=v_2$ such that $\mu_*(\ga'(0))=\de'(0)$. Here $\mu_*$ denotes the derivative of $\mu$. The relation $\sim_*$ is an equivalence relation since $\sim$ is an equivalence relation and the push-forward is a functor. We define $\Te\in\Om^1(S)$ by $\Te_{[k]}([v]):=\te_{k}(v)$. This is well-defined: For any $v_1\sim_* v_2$ and $k_1\sim k_2$, we see that \beqs
\Te_{[k_2]}([v_2])=\te_{k_2}(v_2)=\mu^*\te_{k_1}(v_1)=\te_{k_1}(v_1)=\Te_{[k_1]}([v_1]).
\eeqs  The third equality follows from the fact that $\mu$ is a connection preserving isomorphism. It is easy to see that $R_{r}^*\Te=\Ad_{r^{-1}}\Te$ and $\Te(\rho(r))=r$ for $r\in \R$. Here $\rho_{[k]}(r)=\big([k]\cdot\exp(t\cdot r)\big)'(0)$. We also note that $\Te$ is $G$-invariant: For any $g\in G$ and $v\in T_k K$ such that $v=\ga'(0)$, \beqs (L_g^*\Te)_{[k]}([v])&=\Te_{[g\cdot k]}\left(\left.\frac{d}{dt}\right|_{t=0} g\cdot [\ga(t)]\right)=\Te_{[g\cdot k]}([{L_g}_*(v)])\\&=\te_{g\cdot k}({L_g}_*(v))=(g^*\te)_{k}(v)\srl{*}=\te_{k}(v)=\Te_{[k]}([v]).
\eeqs The equality $*$ follows from the $G$-invariance of the connection form $\te$. 

We verify that $\zeta^*(S,\Te)$ is isomorphic to $(K,\te)$. First we note that the map \beqs
K&\srl{\vph}\ra \zeta^*S\\
k&\mapsto (\pi_K(k),[k])
\eeqs is an isomorphism of $G$-equivariant $S^1$-bundles. Now we see that $\vph^*\wtl\zeta^*\Te=\te$. Here $\wtl\zeta$ is a $G$-equivariant $S^1$-bundle map $\zeta^*S\ra S$ covering $\zeta$. For $k$ and $v$ as above, \beqs (\vph^*\wtl\zeta^*\Te)_{k}(v)=(\wtl\zeta^*\Te)_{(\pi_K(k),[k])}([v])=\Te_{[k]}([v])=\te_{k}v.
\eeqs It is readily seen that the diagram as in \eqref{EQN.Decent.Category.Morphism.Condition} is commutative. Therefore the functor $\zeta^*$ is essentially surjective.

The functor $\zeta^*$ is obviously faithful. We show that it is full. For any $\al:((\zeta^*S_1,\wtl\zeta_{S_1}^*\te_1),1 ) \ra ((\zeta^*S_2,\wtl\zeta_{S_2}^*\te_2),1)$ where $\wtl\zeta_{S_i}^*:\zeta^*S_i\ra S_i$ is the map covering $\zeta$, there exists some $\be \in\Bun(M)(S_1,S_2)$ such that $\zeta^*\be=\al$ because of the above-mentioned special case in \cite{Bry93}. Since $\al$ is $G$-equivariant, it follows that $\be\in G\GBun(M)(S_1,S_2)$. From $\zeta$ being a surjective submersion, for any $v\in T_s S_1$ with $s\in S_1$, there exists $\wtl v\in T_{(z,s)} \zeta^*S_1$ with $\zeta(z)=\pi_{S_1}(s)$ such that $(\wtl\zeta_{S_1})_*\wtl v=v$. Hence we see that $
{\te_1}_s(v)=(\wtl\zeta_{S_1}^*\te_1)_{(z,s)}(\wtl v)=(\al^*\wtl\zeta_{S_2}^*\te_2)_{(z,s)}(\wtl v)=(\wtl\zeta_{S_1}^*\be^*\te_2)_{(z,s)}(\wtl v)=(\be^*\te_2)_s(v)
$ and thus $\be\in G\GBunc(M)\big((S_1,\te_1),(S_2,\te_2)\big)$.
\end{proof}

We need two lemmas which are enhancements of \cite[Lemmas 1.3 and 1.5]{Waldorf-More} in the $G$-equivariant case with connection.  Their proofs follow line-by-line from Waldorf.

\begin{lemma}\label{LEM.Waldorf.Lemma1.3.Geq.conn} Let $\widehat{\cL}=(Y, L, \na, B, \mu)\in G\GGrbc(M)$. There is a canonical isomorphism  $t_\mu:\Delta^*L\ra 1$ of $G$-equivariant $S^1$-bundles over $Y$ satisfying $\pi_1^*t_\mu\tsr \Id =\Delta^*_{112}\mu$ and $\Id \tsr\pi_2^*t_\mu=\Delta_{122}^*\mu$. Here $\Delta:Y\ra Y\times_M Y$ is the diagonal map, and $\Delta_{112},\Delta_{122}:Y^{[2]}\ra Y^{[3]}$ duplicate the first and the second factors, respectively. 
\end{lemma}

\begin{lemma}\label{LEM.Waldorf.Lemma1.5.Geq.conn}
For any $\widehat{\cK}=(Z\srl\zeta\ra P:=Y_1\times_M Y_2, K, \na_K,  \al)\in G\GGrbc(M)\big(\widehat{\cL}_1,\widehat{\cL}_2\big)$, there exists a canonical isomorphism $d_{\widehat\cK}:\zeta_1^*(K,\na_K)\ra \zeta_2^*(K,\na_K)$ such that $\big((K,\na_K), d_{\widehat{\cK}}\big)\in G\GBunc^{\desc}(\zeta)$ and the diagram \[\xymatrix{
(L_1,\na_1)\tsr \zeta_3^*(K,\na_K) \ar[r]^{\zeta_{13}^*\al} \ar[d]_{\Id\tsr \zeta_{34}^*d_{\widehat{\cK}}} & \zeta_1^*(K,\na_K)\tsr (L_2,\na_2) \ar[d]^{\zeta_{12}^*d_{\widehat{\cK}}\tsr\Id}\\
(L_1,\na_1)\tsr \zeta_4^*(K,\na_K) \ar[r]^{\zeta_{24}^*\al} & \zeta_2^*(K,\na_K)\tsr (L_2,\na_2)
}\] is commutative over $(Z\times_P Z)\times_M(Z\times_P Z)$.
\end{lemma}

\begin{prop}\label{Prop:StableIso}The natural inclusions 
\begin{align*} G\GGrb^{\FP}(\cL_1, \cL_2) &\xrightarrow{\simeq} G\GGrb(\cL_1, \cL_2), \\ 
G\GGrbc^{\FP}(\cLh_1, \cLh_2) &\xrightarrow{\simeq} G\GGrbc(\cLh_1, \cLh_2), \end{align*}
are equivalences of categories.
\end{prop}

\begin{proof}
We verify that the functor defined on $G\GGrbc^{\FP}(\cLh_1, \cLh_2)$ is essentially surjective. Given any $\widehat{\cK}=(Z, K,\na_K, \al)\in G\GGrbc(M)\big(\widehat{\cL}_1,\widehat{\cL}_2\big)$ with $G$-cover $Z\srl\zeta\ra P:=Y_1\times_M Y_2$, by Lemma \ref{LEM.Waldorf.Lemma1.5.Geq.conn}, it follows that $\al$ is a morphism of the category $G\GBunc^{\desc}(\zeta^2)$ where $\zeta^2:Z\times_M Z\ra P\times_M P$ is a $G$-cover. By Lemma \ref{Lem:Descent} there exists $(S,\na_S)\in G\GBunc(P)$ such that $\be:\zeta^*(S,\na_S)\srl{\isom}\ra (K,\na_K)$ and since $\al$ is a morphism of $G\GBunc^{\desc}(\zeta^2)$, by the same lemma there exists a morphism $\si$ in $G\GBunc(P^{[2]})$ that is pulled back to $\al$. Now we define a $1$-morphism $\mathcal{S}_{\widehat{\cK}}\in G\GGrbc(M)\big(\widehat{\cL}_1,\widehat{\cL}_2\big)$ by $(P, S,\na_S,\si)$ with $G$-cover $\text{id}_P$. Then there is a $2$-morphism $\mathcal{S}_{\widehat{\cK}}\Rightarrow \widehat{\cK}$ defined by $(Z, \be)$ with $G$-cover $\text{id}_Z$ under the identification  $Z\times_P P\isom Z$. The inclusion functor is clearly faithful, and we show that it is full. Let $\be:\widehat{\cK}_1\Rightarrow\widehat{\cK}_2$ be a $2$-morphism $(W,\be_W)$ with $G$-cover $W\srl{\om}\ra P$. Since $Z\times_P Z\embed Z\times_M Z$ is $G$-equivariant, the compatibility of $\om$ with $\al_i$ of $\widehat{\cK}_i$, over $Z\times_M Z$ implies $\om_1^*\be_W=\om_2^*\be_W$ over $Z\times_P Z$, and hence $\be_W$ is a morphism of $G\GBunc^{\desc}(\om)$. By Lemma \ref{Lem:Descent} there exists a morphism $\be_P$ in $G\GBunc(P)$ that is pulled back to $\be_W$. Accordingly we have a $2$-morphism defined by the pair $(P,\be_P)$ with $G$-cover $\text{id}_P$ which is equivalent to the pair defining $\be$.
\end{proof}


\subsection{Trivializations and isomorphism classes}\label{Subsec:IsoClasses}
Proposition \ref{Prop:StableIso} greatly simplifies the types of morphisms one must consider.  First, the natural equivalence $G\GGrb(\cL_1, \cL_2) \iso G\GGrb(\cI, \cL_1^{-1} \otimes \cL_2)$, given by the Picard structure, tells us that to understand isomorphisms in $G\GGrb(M)$ it suffices to understand isomorphisms $\cI \to \cL$, and thus it suffices to understand the fiber product trivializations $G\GGrb^{\FP}(\cI, \cL)$.

Let us unpackage the data in an arbitrary $\cK = (Z, K, \alpha) \in G\GGrb^{\FP}(\cI, \cL)$.  We may assume the $G$-cover $\zeta \colon Z \to M\times_M Y \iso Y$ is the identity $\Id \colon Y \to Y$.  Continuing, we have a bundle $K \in G\GBun(Y)$ with isomorphism $\pi_2^* K \xrightarrow{\alpha} \pi_1^* K \otimes L$ in $G\GBun(Y^{[2]})$.  Analogous to the description of $\mu$ in Remark \ref{Rem:EquivalentDefn}, the isomorphism $\alpha$ is equivalent to an isomorphism $\upalpha \colon \delta K \to L$ .  The condition that $\alpha$ is compatible with $\mu$ now means that
\[ \begin{tikzcd}{ \delta (\delta K )} \ar[r, "\delta \upalpha"] \ar[rr, bend right, "\iso"] & \delta L \ar[r, "\upmu"  ]  &  1
\end{tikzcd} \]
commutes; i.e. $\upmu \o (\delta \upalpha)$ is the canonical isomorphism $\delta(\delta K) \xrightarrow{\iso} 1$  in $G\GBun(Y^{[3]})$.  A 2-morphism between two such trivializations is an isomorphism $K \to K'$ in $G\GBun(Y)$ compatible with $\alpha, \alpha'$.

Define $G\GGrbtrivc(M)$ as the full subcategory of $G\GGrbc(M)$ consisting of objects $\cLh$ such that the underlying topological gerbe is isomorphic to the trivial gerbe $\cI \iso \cL \in G\GGrb(M)$.  This is the analog of a topologically trivializable $S^1$-bundle whose connection has non-trivial holonomy.

Previously, we discussed how any $B \in \O^2_G(M)$ determines $\cIh_B \in G\GGrbtrivc(M)$ given by the trivial gerbe but with curving 2-form $B$.  We now show that this map is surjective on isomorphism classes of objects.  Note that any $(K,\nabla_K) \in G\GBunc(M)$ naturally determines a 1-morphism $\cIh_B \to \cIh_{B + \curv_G(\nabla_K)}$.  In this way, the groupoid $G\GBunc(M)$ acts on $\O^2_G(M)$, and we denote the induced action 2-groupoid as $\O^2_G(M) \sslash G\GBunc(M)$.  Explicitly, $\O^2_G(M)$ is the set of objects, and the groupoid of morphisms from $B_1$ to $B_2$ is the subgroupoid of $G\GBunc(M)$ given by $(K,\nabla)$ with $\curv_G(\nabla) = B_2-B_1$.

We first point out a basic lemma whose proof is a simple homework problem.

\begin{lemma}\label{Lem:AffineExact}Suppose $G_1 \to G_2 \to G_3$ is an exact sequence of groups, and 
\[ \begin{tikzcd}[row sep=tiny] A_1 \ar[r, "f_1" '] & A_2 \ar[r, "f_2" '] & A_3 \\
G_1 \ar[r] \ar[u, phantom, "\circlearrowleft"] & G_2 \ar[r] \ar[u, phantom, "\circlearrowleft"] & G_3 \ar[u, phantom, "\circlearrowleft"]
\end{tikzcd} \]
are equivariant maps, where each $A_i$ is acted on freely and transitively by $G_i$.  Then, the image  $f_2(f_1(A_1))$ is a single element $* \in A_3$; and, $a_2 \in A_2$ is in the image of $f_1$ if and only if $f_2(a_2)= * \in A_3$.
\end{lemma}

\begin{prop}\label{Prop:Grbtrivc}The natural functor 
\[  \O^2_G(M) \sslash G\GBunc(M) \xrightarrow{\cIh_\bullet} G\GGrbtrivc(M)\] 
is an equivalence of 2-groupoids and induces a natural isomorphism of abelian groups
\[ \frac{\O^2_G(M)}{\O^2_G(M)_\Z} \xrightarrow{\iso} \pi_0\left( G\GGrbtrivc(M) \right). \]
\end{prop}
\begin{proof}
We first we show that the functor is surjective on isomorphism classes of objects by proving that any $\cLh = (Y, L, \nabla^L, B, \mu) \in G\GGrbtrivc(M)$ is isomorphic to $\cIh_\rho$ for some $\rho \in \O^2_G(M)$.  By definition, if $\cLh \in G\GGrbtrivc(M)$ there exists a topological isomorphism $\cI \xrightarrow{\cK} \cL$.  As explained above, we use Proposition \ref{Prop:StableIso} to assume that $\cK$ is given by a bundle $K \in G\GBunc(Y)$ and an isomorphism
\[ \delta K  \overset{\upalpha}\longrightarrow  L \quad \in G\GBun(Y^{[2]}),\] where $\de K$ denotes $  \pi_2^* K \otimes (\pi_1^* K)^{-1}$.  We now show the existence of a $G$-equivariant connection on $K$.  Letting $\cA(-)^G$ denote the affine space of $G$-invariant connections on an equivariant bundle, we have
\[ \begin{tikzcd}[row sep=tiny] &\cA(K)^G \ar[r, "\delta"] &\cA(\delta K)^G \ar[r,"\delta"] & \cA(\delta^2 K)^G \iso \cA(1)^G \\ 
\O^1(M)^G \ar[r]  &\O^1(Y)^G \ar[r, "\delta"] \ar[u, phantom, "\circlearrowleft"] & \O^1(Y^{[2]})^G \ar[r,"\delta"] \ar[u, phantom, "\circlearrowleft"] &\O^1(Y^{[3]})^G \ar[u, phantom, "\circlearrowleft"]  , \end{tikzcd}\]
where the top row is affine over the bottom row, and the bottom row is the exact sequence \eqref{Eq:CoverSeq3}.  By Lemma \ref{Lem:AffineExact}, an element $\nabla^{\delta K} \in \cA(\delta K)^G$ is in the image of $\delta \left(\cA(K)^G\right)$ if and only if it induces the trivial connection on $\delta(\delta K)\iso 1$,  where the canonical isomorphism also equals $\upmu\o\delta \upalpha$.   Pulling back $\nabla^L$ via $\upalpha$ gives $(\delta K, \upalpha^* \nabla^L) \in G\GBunc(Y^{[2]})$.  The natural isomorphism $\delta (L,\nabla^L) \xrightarrow{\upmu} 1$ then implies implies that 
\[ \delta(\upalpha^*\nabla^L) = (\delta \upalpha)^*(\delta \nabla^L) = (\delta \upalpha)^* (\upmu^* 1) = (\upmu \o \delta \upalpha)^* 1. \] 
Hence, there exists a (non-unique) connection $\nabla^K$ such that $(K, \nabla^K) \in G\GBunc(Y)$ and $\delta (K,\nabla^K) \xrightarrow{\alpha} (L,\nabla^L) \in G\GBunc(Y^{[2]})$.  Finally, consider $B - \curv_G(\nabla^K) \in \O_G^2(Y)$.  Since $\delta(B-\curv_G(\nabla^K)) = \curv_G(\nabla^L) - \curv_G(\alpha^* \nabla^L) = 0$, it follows that there exists a unique $\rho \in \O_G^2(M)$ such that $\pi^* \rho = B - \curv_G(\nabla^K)$.  Therefore, there exists an isomorphism $\cKh \colon \cIh_\rho \to \cLh$ given by  $\cKh = (Y,K,\nabla^K, \alpha)$.

We now consider the groupoid of morphisms.  Any $(K,\nabla) \in G\GBunc(M)$ that satisfies $\curv_G(\nabla) = B_2-B_1 \in \O^2_G(M)$ clearly determines a 1-morphism, viewed as an object in the groupoid of morphisms $G\GGrbc(\cIh_{B_1},\cIh_{B_2})$.  To see that this is essentially surjective, we again will use Lemma \ref{Lem:Descent}.  We may assume that an isomorphism $\cIh_{B_1} \to \cIh_{B_2}$ is isomorphic to one given by a $G$-equivariant $S^1$-connection $(K,\nabla)\in G\GBunc(M)$ satisfying $\curv_G(\nabla) = B_2 - B_1$, together with an isomorphism $\alpha\colon (K,\nabla) \to (K,\nabla)$.  The compatibility of $\alpha$ with $\mu_1=\mu_2 = \Id$ implies that $\alpha \o \alpha =\alpha$, and hence $\alpha = \Id$ (c.f. \cite[Section 3]{Waldorf-More}).  Therefore, any 1-morphism $\cIh_{B_1} \to \cIh_{B_2}$ is equivalent to some $(K,\nabla) \in G\GBunc (M)_{B_2-B_1}$.  Finally, for such 1-morphisms $(K_i,\nabla^i)$, the 2-morphisms $(K_1,\nabla^1) \Rightarrow (K_2, \nabla^2)$ in $G\GGrbc^{\FP}(\cIh_{B_1}, \cIh_{B_2})$ are precisely isomorphisms in $G\GBunc(M)$.

Finally, the image of $\curv_G \colon G\GBunc(M) \to \O^2_G(M)$ equals $\O^2_G(M)_\Z$.  This fact, which we believe is fairly well-known, is implied by the isomorphism $\pi_0 \left(G\GBunc(M) \right) \iso \HG^2(M)$  \cite[Proposition 5.10]{Redden16a} and the fact that $\HG^2(M) \to \O^2_G(M)_\Z$ is surjective. 
\end{proof}

\begin{cor}\label{Cor:Aut}
The natural functors 
\begin{align*}
G\GBunc(M)_{\fl} &\xrightarrow{\simeq} \Aut (\cLh),\\
G\GBun(M) &\xrightarrow{\simeq} \Aut (\cL), 
\end{align*} are equivalences of groupoids.
\end{cor}
\begin{proof}
As noted in \eqref{Eq:Automorphisms}, the Picard structure gives a natural identification $\Aut(\cLh) \iso \Aut(\cIh_0)$ for any $\cLh \in G\GGrbc(M)$.  By Proposition \ref{Prop:Grbtrivc}, $\Aut(\cIh_0)$ is equivalent to the full subgroupoid of $G\GBunc(M)$ consisting of objects $(K,\nabla)$ with equivariant curvature $\curv_G(\nabla) =0-0=0$, which is the groupoid $G\GBunc(M)_{\fl}$.   Similarly, there is a natural identification $\Aut(\cL) \iso \Aut(\cI)$ for any $\cL \in G\GGrb(M)$.  The same argument used in Proposition \ref{Prop:Grbtrivc}, but ignoring the differential forms and connection data, shows $\Aut(\cI)$ is equivalent to $G\GBun(M)$.
\end{proof}

\begin{prop}\label{Prop:DeRhamClass}Let $\cLh_1, \cLh_2 \in G\GGrbc(M)$.  If there is an isomorphism of the underlying equivariant bundle gerbes without connection $\cL_1 \xrightarrow{\cK} \cL_2 \in G\GGrb(M)$, then 
\[ [\curv_G(\cLh_1)] = [\curv_G(\cLh_2)] \in H^3_G(M;\R). \]
\end{prop}
\begin{proof}
Let $\cK\colon \cL_1 \to \cL_2$ be an isomorphism, which is is equivalent to an isomorphism $\cI \to \cL_1^{-1} \otimes \cL_2$.  By Proposition \ref{Prop:Grbtrivc}, there is an isomorphism $\cIh_B \to \cLh_1^{-1} \otimes \cLh_2$ for some $B \in \O^2_G(M)$, and thus
\[ d_G B = \curv_G (\cLh_2) - \curv_G(\cLh_1) \in \O^3_G(M).  \qedhere \]
\end{proof}

The above discussion may be summarized by the following proposition, which is given by combining  Propositions \ref{Prop:Curvature}, \ref{Prop:Grbtrivc}, \ref{Prop:DeRhamClass} and \cite[Proposition 3.2]{Sti10}.

\begin{prop}\label{Prop:StienonProperties}
The following is a commutative diagram of abelian groups
\begin{equation}\label{Eq:StienonProperties} \vcenter{ \xymatrix{ \pi_0 \left(G\GGrbc(M)\right) \ar@{->>}[r] \ar[d]^{\curv_G}&  \pi_0 \left(G\GGrb(M) \right)\ar[d] \\
\O^3_G(M)_{\operatorname{closed}} \ar[r] & H^3_G(M;\R).} }  \end{equation}
The upper horizontal map is surjective and determines the short exact sequence
\[ 0 \to \frac{\O^2_G(M)}{\O^2_G(M)_\Z} \xrightarrow{\cIh_{\bullet}}  \pi_0 \left(G\GGrbc(M)\right) \to  \pi_0 \left( G\GGrb(M) \right)  \to 0.\]
\end{prop}

\subsection{The induced simplicial gerbe}
As observed in numerous sources (e.g.  \cite{Mein-Grb, MRSV15j, NikSch-Equiv, Sti10, TuXu15}), an equivariant gerbe $\cL \in G\GGrb(M)$ naturally determines a gerbe $G^\bullet \times \cL$ on the simplicial manifold $G^\bullet \times M$.  We write this as
\begin{equation}\label{Eq:SimplicialGerbe} \begin{tikzcd} \cdots  \quad G\times G \times \cL \ar[r, "\wt{d}_1" description] \ar[r, shift left=1.5, "\wt{d}_0"] \ar[r, shift right=1.5, "\wt{d}_2" ' ] \ar[d]& G \times \cL  \ar[r, shift left=.75, "\wt{d}_0" ] \ar[r, shift right=.75, "\wt{d}_1" '] \ar[d]& \cL \ar[d] \\
\cdots  \quad G \times G \times M  \ar[r, "d_1" description] \ar[r, shift left=1.5, "d_0"] \ar[r, shift right=1.5, "d_2" ']& G \times M  \ar[r, shift left=.75, "d_0"] \ar[r, shift right=.75, "d_1" '] & M
\end{tikzcd} \end{equation}
where $d_i$ denote the face maps, with $d_0(g,x) = x$ the source and $d_1(g,x) = g\!\cdot\! x$ the target.  Explicitly, if $\cL = (Y, L, \mu) \in G\GGrb(M)$,  the product gerbe $(G^k\times \cL) \in \Grb(G^k\times M)$ is defined by $(G^k\times Y, G^k \times L, \Id \times \mu)$.  The diagram \eqref{Eq:SimplicialGerbe} is then an abbreviated form of the following diagram.
\begin{equation}\label{Eq:SimpGerbeFullDiag} \begin{tikzcd}[row sep=small, column sep=small]
\cdots &  G\times G \times L \ar[r, "\wt{d}_1" description] \ar[r, shift left=1.5, "\wt{d}_0"] \ar[r, shift right=1.5, "\wt{d}_2" '] \ar[d]& G \times L  \ar[r, shift left=.75, "\wt{d}_0"] \ar[r, shift right=.75, "\wt{d}_1" '] \ar[d] & L \ar[d] \\
\cdots & G \times G \times Y^{[2]} \ar[r] \ar[r, shift left=1.5] \ar[r, shift right=1.5] \ar[d, shift left=.75] \ar[d, shift right=.75]& G \times Y^{[2]} \ar[r, shift left=.75] \ar[r, shift right=.75] \ar[d, shift left=.75] \ar[d, shift right=.75]& Y^{[2]}  \ar[d, shift left=.75] \ar[d, shift right=.75]\\
\cdots & G\times G \times Y \ar[r] \ar[r, shift left=1.5] \ar[r, shift right=1.5] \ar[d]& G \times Y  \ar[r, shift left=.75] \ar[r, shift right=.75] \ar[d]& Y \ar[d] \\
\cdots & G \times G \times M \ar[r, "d_1" description] \ar[r, shift left=1.5, "d_0"] \ar[r, shift right=1.5, "d_2" '] & G \times M  \ar[r, shift left=.75, "d_0"] \ar[r, shift right=.75, "d_1" '] & M
\end{tikzcd}\end{equation}

The maps $\wt{d}_i \colon G^{k+1} \times \cL \to G^k \times \cL$ determine a particularly strong form of gerbe isomorphism  $d_i^*(G^k\times \cL) \iso (G^{k+1}\times \cL) \in \Grb(G^{k+1}\times M)$ where all structures are diffeomorphic; one does not need to first pass to another cover.

\section{Gerbes on the quotient stack}\label{Sec:Stack}

We now give a different and more conceptual view of $G$-equivariant gerbe connections.  As a simpler example, it is well-known that $G$-equivariant $S^1$-bundles on $M$ are equivalent to $S^1$-bundles on the Lie groupoid $G\times M \rightrightarrows M$, which are equivalent to $S^1$-bundles on the simplicial manifold $G^{\bullet}\times M$.  When connections are introduced, the situation is more complicated.  Given $P \in G\GBun(M)$, a $G$-invariant connection on $P$ will not define a connection on the induced bundle over the Lie groupoid $G\times M \rightrightarrows M$ unless that connection was $\fg$-basic.  However, when one realizes that $\E G\times_G M$ is the globally valid replacement of $M\sslash G$, it becomes natural to consider $\En G \times_G M$, the version of the quotient stack that includes connections on the principal $G$-bundles.

We {\it define} $\Grbc(\En G \times_G M) \= \Shv_{\infty}(\En G \times_G M, \BBn S^1)$.  In other words, an element $\cEllh \in \Grbc(\En G \times_G M)$ assigns to every test manifold $X$ a functor 
\[ (\En G\times_G M)(X) \longrightarrow  (\BBn S^1) (X) = \Grbc(X),\]
and it does so in a natural way compatible with pullback maps.  Explicitly, a map of $G$-bundles $\varphi$ gives an isomorphism between the two induced gerbes with connection, as indicated in the following diagram.
\begin{equation}\label{Eq:Functoriality} \begin{tikzcd}[column sep=small]
(P', \varphi^*\Theta) \ar[r, "\varphi"] \ar[d] & (P,\Theta) \ar[r,"f"] \ar[d] & M \\
X' \ar[r, "\bar{\varphi}"] & X
\end{tikzcd} \!\! \mapsto \,
\cEllh(P', \varphi^*\Theta, f\o \varphi) \xrightarrow{\cEllh(\varphi)} 
 \bar{\varphi}^* \cEllh(P,\Theta,f)   \in \Grbc(X')   \end{equation}
We similarly define $\Grb(\E G \times_G M) \= \Shv_\infty(\E G\times_G M, \B^2 S^1)$.  Fortunately, isomorphism classes of such gerbes may be classified via existing results and general theory.

\begin{prop}\label{Prop:H3GGrb}There are natural isomorphisms of abelian groups
\begin{align*} 
\pi_0( \Grb(\E G \times_G M)) &\xrightarrow{\iso}  H^3_G(M;\Z),\\
\pi_0 (\Grbc(\En G \times_G M)) &\xrightarrow{\iso}  \HG^3(M;\Z). \end{align*}
\end{prop}
\begin{proof}
The second isomorphism was already explained in Section \ref{Subsec:Coh}.  The Deligne complex gives a model for the 2-category of gerbes with connections, and hence there is an equivalence of sheaves 
\[ \BBn S^1 \simeq \L(\DK(S^1\xrightarrow{d}\O^1\xrightarrow{d}\O^2))\simeq \Kh(\Z,3),\]
which induces a natural isomorphism
\[ \pi_0 \Shv_\infty(\En G \times_G M, \BBn S^1) \iso \pi_0 \Shv_\infty(\En G \times_G M, \cKh(\Z,3)) =: \HG^3(M).\]

To see the corresponding result without connections, we know that the sheaf of bundle gerbes $\B^2 S^1$ is equivalent to $\L(\DK(S^1\to0\to0))$.  So, for any sheaf $\cM \in \Shv_\infty$ we have
\[ \pi_0 \Grb(\cM) \= \pi_0 \Shv_\infty( \cM, \B^2 S^1) \iso H^2_{\Shv}(\cM; S^1).\]
The short exact sequence of abelian Lie groups $0 \to \Z \into \R \to S^1\to 0$ induces a long  exact sequence 
\[  \cdots \to H^2_{\Shv}(\cM; \O^0) \to H^2_{\Shv}(\cM; S^1) \to H^3_{\Shv}(\cM; \Z) \to H^3_{\Shv}(\cM; \O^0) \to \cdots.  \]
Here we write $\O^0$ to indicate that $\R$ is the sheaf of smooth real-valued functions, as opposed to the sheaf $\R^\delta$ of locally constant functions.  By Lemma \ref{Lem:Flabby} below, which uses that $M \sslash G$ is a proper Lie groupoid, the first and last terms in the above sequence vanish and give an isomorphism 
\[ 0 \to H^2_{\Shv}( \E G\times_G M, S^1) \xrightarrow{\iso} H^3_{\Shv}( \E G\times_G M, \Z) \to 0. \]
Finally, the sheaf $\Z$ is discrete, and hence 
\[H^3_{\Shv}(\E G\times_G M;\Z) \iso \pi_0 \Top (EG \times_G M; K(\Z,3)) = H^3_G(M;\Z), \] 
as explained in \eqref{Eq:EquivCoh}.
\end{proof}

The following lemma, used in the proof of Proposition \ref{Prop:H3GGrb}, is a special case of one of the main theorems from \cite{MR3107517}, though their language looks different than ours.  For that reason, we prove why our desired result follows from the theorem of Abad--Crainic.

\begin{lemma}[Corollary 4.2 of \cite{MR3107517}]\label{Lem:Flabby}Suppose that $G$ is a compact Lie group.  If $n>0$, then $H^n_{\Shv}(\E G\times_G M, \O^0) = 0$.
\end{lemma}
\begin{proof}If $G$ is compact, then the action of $G$ on $M$ is proper and $G\times M \rightrightarrows M$ is a proper Lie groupoid.  Thus, Corollary 4.2 of  \cite{MR3107517}  implies the (co)homology of
\[  \O^0(M) \xrightarrow{\delta} \O^0( G\times M) \xrightarrow{\delta} \O^0(G^2 \times M) \xrightarrow{\delta} \cdots \xrightarrow{\delta} \O^0( G^{n}\times M) \xrightarrow{\delta} \cdots \]
vanishes at all terms where $n>0$.

Our definition of sheaf cohomology in \eqref{Eq:SheafCohDef} states that 
\[  H^n_{\Shv} (\E G \times_G M, \O^0) \= \pi_0 \Shv_\infty\left(\E G\times_G M, \L(\DK( \O^0[-n]))\right).\]
One can actually calculate this by the following: 
\begin{align*} \Shv&_\infty(\E G \times_G M, \L(\DK(\O^0[-n]))) \simeq \Shv_\infty( \L(M \sslash G), \L(\DK(\O^0[-n]))) \\
&\simeq \Pre_\infty( M \sslash G, \L(\DK(\O^0[-n]))) 
\simeq \Pre_\infty( M \sslash G, \DK(\O^0[-n])) \\
&\simeq \Pre_\infty( \hocolim_{\Delta^\op} G^\bullet \times M, \DK(\O^0[-n])) 
\simeq \holim_{\Delta^\op} \Pre_\infty ( G^\bullet \times M, \DK(\O^0[-n])) \\
&\simeq \holim_{\Delta^\op}  \DK(\O^0(G^\bullet \times M)[-n])  
 \simeq \DK \big(  \cdots \xrightarrow{\delta} \O^0(G^{n-1}\times M) \xrightarrow{\delta} \O^0( G^{n}\times M)_{\delta\text{-cl}}  \big).
 \end{align*} 
 
The above string of equivalences is justified by the following.  The first equivalence follows from $\L(M\sslash G) \simeq \E G\times_G M$, as noted in Remark \ref{Rem:ActionGpd}.  The second is because sheafification is a left-adjoint.  For the third equivalence, the presheaf $\O^0 \in \Pre_{\Set}$ is flabby, so $\DK(\O^0[-n]) \in \Pre_\infty$ satisfies the descent property of $\Shv_\infty$ and $\L(\DK(\O^0[-n])) \simeq \DK(\O^0[-n]))$.  To see the fourth, for any simplicial set $B$ there is a natural equivalence $\hocolim_{\Delta^\op} N(S) \simeq B$, where we use the composition $\Delta^{\op} \xrightarrow{S}\Set \xrightarrow{N} \sSet$ to consider $N(S)$ as a simplicial object in simplicial sets.  The fifth is a general property of colimits, and the sixth is Yoneda's lemma.  The seventh is a simplification of the general principal that in chain complexes, homotopy limits over a simplicial diagram may be calculated via the total double complex; see \cite[Section 19.8]{DuggerHoCoLim} or \cite[Appendix B.1]{MR3376591}.

Finally, $\pi_0$ of the resulting simplicial set is simply $\O^0(G^n\times M)_{\delta\text{-cl}} / \delta(\O^0(G^{n-1} \times M))$, which is precisely what is proven to vanish by Abad--Crainic.
\end{proof}

\subsection{The functor between the two models}
Our goal is now to relate these two models of equivariant gerbe connections, given by $G\GGrbc(M)$ and $\Grbc(\En G \times_G M)$, in a way analogous to \eqref{Eq:EquivConns}.  The following proposition will be the key to doing so, and it generalizes the construction of \eqref{Eq:ConnsQuot}:  given $(P,\Theta) \in \Bunc[G](X)$ and $(L,\nabla)\in G\GBunc(P)$, there is a canonical way to modify $\nabla$ so that it becomes $G$-basic and descends to the quotient $L/G$.

\begin{prop}\label{Prop:Quotient} \
\begin{enumerate}
\item A principal $G$-bundle $P \in \Bun[G](X)$ determines a natural functor
\begin{align*} G\GGrb(P) &\xrightarrow{/G}  \Grb(X) \\
\cL &\longmapsto \cL/G.\end{align*}
\item A principal $G$-bundle with connection $(P,\Theta) \in \Bunc[G](X)$ determines a natural functor
\[ G\GGrbc(P) \xrightarrow{\Theta^*} \Grbc(X), \]
and the curvature satisfies
\[  \curv(\Theta^* \cLh) = \Theta^*(\curv_G(\cLh)) \in \O^3(X) \]
under $\O^3_G(P) \xrightarrow{\Theta^*} \O^3(X)$.  
\end{enumerate}
\end{prop}
\begin{proof}
We first do the topological case, which is straightforward.  Let $\cL = (Y, L, \mu) \in G\GGrb(P)$.  Since $G$ acts freely on $P$, it also acts freely on $Y$, $Y^{[2]}$, and $L$.  Therefore, we have the following commutative diagram, 
\begin{equation} \label{Eq:QuotDiag} \begin{tikzcd}
L  \ar[r] \ar[d] & Y^{[2]} \ar[d] \ar[r, shift left, "\pi_1"] \ar[r, shift right, "\pi_2" '] \ar[d] &Y \ar[d] \ar[r, "\pi"] & P \ar[d]\\
L/G  \ar[r] & (Y^{[2]})/G \ar[r, shift left, "\bar{\pi}_1"] \ar[r, shift right, "\bar{\pi}_2" '] &Y/G  \ar[r, "\bar{\pi}"] & P/G=X
\end{tikzcd} \end{equation}
where each vertical map is a principal $G$-bundle.   

Because $G$ acts freely on $P$ and $Y$, there is a natural diffeomorphism 
\begin{equation}\label{Eq:QuotY2} (Y^{[2]})/G = (Y\times_P Y) /G  \iso (Y/G) \times_{P/G} (Y/G) = (Y/G)^{[2]}.\end{equation}
To see this, observe that the projection $Y \times_P Y \to (Y/G) \times_{P/G} (Y/G)$ fits into the commutative diagram
\[ \begin{tikzcd}[cramped, sep=tiny]
Y^{[2]} \ar[rr] \ar[dd] \ar[dr] && Y \ar[dd]  \ar[dr] \\
& (Y/G)^{[2]}  \ar[rr, crossing over] && Y/G \ar[dd] \\
Y \ar[rr] \ar[dr] && P \ar[dr] \\
&Y/G \ar[rr] \ar[from=uu, crossing over] && P/G ,
\end{tikzcd} \]
where the front and rear faces are pullbacks of surjective submersions, and the other four faces are  pullback squares of principal $G$-bundles.  This gives $Y^{[2]} \to (Y/G)^{[2]}$ the structure of a principal $G$-bundle and produces the desired diffeomorphism \eqref{Eq:QuotY2}. 

The original bundle gerbe multiplication $\mu$ is $G$-equivariant and naturally descends to a map on the quotients 
\[ \bar{\pi}_{12}^*(L/G) \otimes \bar{\pi}_{23}^*(L/G) \xrightarrow{\mu/G} \bar{\pi}_{13}^* (L/G),\]
where $\bar{\pi}_{ij}\colon (Y/G)^{[3]} \to (Y/G)^{[2]}$ denotes the induced map on the quotient manifolds.  Therefore, the bottom row of \eqref{Eq:QuotDiag} is the data of the bundle gerbe on $P/G=X$,
\[ \cL/G \= (Y/G, L/G, \mu/G) \in \Grb(X).\]
Similarly, any $G$-equivariant 1-morphism $\cK = (Z, K, \alpha)$ or 2-morphism $\cJ = (W,\beta)$ descends to an ordinary bundle gerbe 1-morphism $\cK/G \= (Z/G, K/G, \alpha/G)$ or 2-morphism $\cJ/G \= (W/G, \beta/G)$.  This establishes the desired functor $G\GGrb(P) \to \Grb(P/G)$ when $G$ acts freely on $P$.

For the second part, assume $(P,\Theta) \in \Bunc[G](X)$.  We define
\[ \Theta^*\cLh \= (Y/G, \Theta^*(L,\nabla), \Theta^*B, \mu/G) \in \Grbc(X), \]
where $\Theta^*(L,\nabla) \in \Bunc((Y/G)^{[2]})$ and $\Theta^*B \in \O^2(Y/G)$ are given  by \eqref{Eq:Theta*} and \eqref{Eq:WeilHomo}, respectively.  The compatibility between $\Theta^*B$ and $\Theta^*(L,\nabla)$, along with the calculation of $\curv(\Theta^*\cLh)$, follows directly from the naturality of the Weil homomorphism.  First,
\[ \curv\big( \Theta^* \nabla  \big) = \Theta^* (\curv_G(\nabla)) = \Theta^*(\pi_2^*B - \pi_1^* B) = \bar{\pi}_2^* (\Theta^*B) - \bar{\pi}_1^*(\Theta^* B) \in \O^2((Y/G)^{[2]}). \]
Second, since 
\[ \begin{tikzcd}
\O^2_G(Y) \ar[r, "\Theta^*"] \ar[d, "d_G" '] & \O^2(Y/G) \ar[d, "d"] \\
\O^3_G(Y) \ar[r, "\Theta^*"] & \O^3(Y/G), 
\end{tikzcd} \]
we immediately obtain
\[ \bar{\pi}^* \curv( \Theta^* \cLh) = d (\Theta^* B) = \Theta^* d_G B = \Theta^* (\pi^*\curv_G(\cLh)) = \bar{\pi}^*(\Theta^*\curv_G(\cLh)).\]
The map $\bar{\pi}^*$ is injective, therefore $\curv(\Theta^* \cLh) = \Theta^* \curv_G (\cLh)$.  

Similarly, any $G$-equivariant 1-morphism with connection $\cKh = (Z, K, \nabla^K, \alpha)$ is mapped to $\Theta^* \cKh \= (Z/G, \Theta^*(K,\nabla^K), \alpha/G)$.  For a 2-morphism, $\cJh$ contains no additional data past $\cJ$, it merely satisfies further requirements.
\end{proof}

There is now an obvious map 
\begin{align}\label{Eq:FundMap1}G\GGrbc(M) &\longrightarrow \Grbc(\En G \times_G M) \\
\cLh &\longmapsto \cEllh, \nonumber
\end{align}
given by the following.  For a moment, we use $\phi$ to denote the connection 1-form so as not to confuse between connection forms and covariant derivatives.  If $\cLh = (Y, L, \phi, B, \mu)$, then for $(P,\Theta, f) \in (\En G \times_G M)(X)$ use the composition 
\[ G\GGrbc(M) \xrightarrow{f^*} G\GGrbc(P) \xrightarrow{\Theta^*} \Grbc(X)\]
to define 
\begin{equation}\label{Eq:FundMap2} \cEllh (P,\Theta, f) \=  \Theta^*(f^* \cLh) \in \Grbc(X).
\end{equation}
Explicitly, this can be written as
\[  \cEllh (P,\Theta, f) = \big((f^*Y)/G,\> (f^*L)/G,\> (1-\iota_\Theta)f^*\phi,\> \Theta^*(f^*B),\> (f^*\mu)/G\big), \]
as seen through the diagram 
\[ \begin{tikzcd}(P,\Theta) \ar[r,"f"] \ar[d] & M \\ X \end{tikzcd}  \mapsto
 \begin{tikzcd}[row sep=small]
(f^*L,\> (1-\iota_{\Theta})f^*\phi) \ar[r] \ar[d] & (L,\>\phi) \ar[d] \\
f^* (Y^{[2]}) \ar[r] \ar[d, shift left] \ar[d, shift right]& Y^{[2]} \ar[d, shift left] \ar[d, shift right] \\
 f^* Y \ar[r] \ar[d] & Y \ar[d] \\ 
(P,\Theta) \ar[r, "f"] \ar[d] & M \\ X
\end{tikzcd} \mapsto 
\begin{tikzcd}[row sep=17]
(f^*L/G, \>(1-\iota_{\Theta})f^*\phi) \ar[d] \\
((f^*Y)/G)^{[2]} \ar[d, shift left] \ar[d, shift right] \\
 (f^* Y) /G \ar[d]  \\ 
X
\end{tikzcd} \]
Similarly, we define 
\begin{align}\label{Eq:FundMapNoConn}G\GGrb(M) &\longrightarrow \Grb(\E G \times_G M) \\
\cL &\longmapsto \cEll \nonumber
\end{align}
by $\cEll(P,f) \= (f^*\cL)/G \in \Grb(X)$ for $(P,f) \in (\E G \times_G M)(X)$.

Using the functors \eqref{Eq:FundMap1} or \eqref{Eq:FundMapNoConn}, along with Proposition \ref{Prop:H3GGrb}, we now define the {\it equivariant Dixmier--Douady class} $\DD_G(\cLh) = \DD_G(\cL) \in H^3_G(M;\Z)$ as the image of $\cLh$ or $\cL$ in the composition
\begin{equation}\label{EqDefn:DD} \pi_0 \left(G\GGrbc(M)\right) \to \pi_0 \left(G\GGrb(M)\right)  \to \pi_0 \left( \Grb(\E G\times_G M) \right) \overset{\iso}\to H^3_G(M;\Z).\end{equation}

Our present goal is to show the second map in \eqref{EqDefn:DD} is an isomorphism.  Fortunately, Tu--Xu  already proved that the homomorphism $\pi_0 (G\GGrb(M))  \to H^3_G(M;\Z)$ is surjective.  In  \cite[Theorem 2.8, Corollary 2.10]{TuXu15}, they show that for any class $\xi \in H^3_G(M;\Z)$, there exists some $\cL \in G\GGrb(M)$ with $\DD_G(\cL) = \xi$.

It therefore remains for us to show the homomorphism $\pi_0 (G\GGrb(M)) \to H^3_G(M;\Z)$ is injective, a fact that was previously noted but not proven in Remark 5.8 of \cite{Sti10}.  Our proof relies on better understanding the behavior of $\cL \mapsto \cEll$ and using this to show that $\pi_0(G\GGrb(M)) \to \pi_0(\Grb(\E G \times_G M))$ is an injection.

\begin{prop}\label{Prop:SimplicialGerbe}Let $\cL \in G\GGrb(M)$.  The image of $\cL$ in the composition
\[ G\GGrb(M) \to \Grb(\E G\times_G M) \to \Grb(M\sslash G)\]
is isomorphic to $G^\bullet \times \cL \in \Gpdinfty$, as defined in \eqref{Eq:SimplicialGerbe}.
\end{prop}
\begin{proof}
The simplicial manifold $G^\bullet \times M = M\sslash G = N(G\times M \rightrightarrows M)$ has a  natural map to $\E G \times_G M$ induced by the trivial $G$-bundle.  First, consider $M \to \E G \times_G M$ defined by 
\[ \begin{tikzcd}
G\times M \ar[r, "d_1"] \ar[d, "d_0" '] & M  \\
M, \ar[u, dashed, bend right, "s_0" '] 
\end{tikzcd} \]
where we view $G \times M$ as a trivial principal (left) $G$-bundle over $M$, and $d_1\colon G \times M \to M$ as the $G$-equivariant map $(g,m)\mapsto g\cdot m$; here $G$ acts on $G\times M$ by left multiplication on $G$ only.  The bundle also has a natural section of $d_0$ given by the degeneracy map $s_0 \colon M \to G\times M$, since $d_0 s_0 = \Id$.  Thus, the value of $\cEll$ on $d_1$ is given by 
\[ \cEll(G\times M, d_1) = (d_1^* \cL)/G \iso s_0^* (d_1^* \cL) \iso (d_1 \o s_0)^* \cL \iso \cL \in \Grb(M).  \]

The remaining components of the natural map $M\sslash G \to \E G \times_G M$ can be seen through the following.
\begin{equation}\label{Eq:SimpAct} \begin{tikzcd} \cdots &[-20]G \times G \times G \times M \ar[d, "d_0" '] \ar[r, "d_2" description] \ar[r, shift left=1.5, "d_1"] \ar[r, shift right=1.5, "d_3" ']& G \times G \times M \ar[d,"d_0" '] \ar[r, shift left=.75, "d_1"] \ar[r, shift right=.75, "d_2" '] & G \times M \ar[r, "d_1"] \ar[d, "d_0" ']& M \\
\cdots &G \times G \times M  \ar[r, "d_1" description] \ar[r, shift left=1.5, "d_0"] \ar[r, shift right=1.5, "d_2" '] \ar[u, dashed, bend right]& G \times M \ar[r, shift left=.75, "d_0"] \ar[r, shift right=.75, "d_1" '] \ar[u, dashed, bend right]& M \ar[u, dashed, bend right, "s_0" ']
\end{tikzcd} \end{equation}
This diagram commutes because of the simplicial identities $d_0 d_j = d_{j-1} d_0$ for $j>0$, and hence it forms a simplicial $G$-bundle on $G^\bullet \times M$ with an equivariant map to $M$.

To compute $\cEll$ when evaluated on $M\sslash G$, note that we have a natural isomorphism between the pullback of $\cL$ in \eqref{Eq:SimpAct} and the following diagram.
\[ \begin{tikzcd}\cdots&[-20] G \times G \times G \times L \ar[d] \ar[r, "\wt{d}_2" description] \ar[r, shift left=1.5, "\wt{d}_1"] \ar[r, shift right=1.5, "\wt{d}_3" ']& G \times G \times L \ar[d] \ar[r, shift left=.75, "\wt{d}_1"] \ar[r, shift right=.75, "\wt{d}_2" '] & G \times L \ar[r, "\wt{d}_1"] \ar[d]& L \ar[d, "\pi" ] \\
\cdots & G \times G \times G \times Y^{[2]} \ar[d, shift left=.75] \ar[d, shift right=.75] \ar[r, "\wt{d}_2" description] \ar[r, shift left=1.5, "\wt{d}_1"] \ar[r, shift right=1.5, "\wt{d}_3" ']& G \times G \times Y^{[2]} \ar[d, shift left=.75] \ar[d, shift right=.75] \ar[r, shift left=.75, "\wt{d}_1"] \ar[r, shift right=.75, "\wt{d}_2" '] & G \times Y^{[2]} \ar[r, "\wt{d}_1"] \ar[d, shift left=.75] \ar[d, shift right=.75] & Y^{[2]} \ar[d, shift left=.75] \ar[d, shift right=.75] \\
\cdots & G \times G \times G \times Y \ar[d] \ar[r, "\wt{d}_2" description] \ar[r, shift left=1.5, "\wt{d}_1"] \ar[r, shift right=1.5, "\wt{d}_3" ']& G \times G \times Y \ar[d] \ar[r, shift left=.75, "\wt{d}_1"] \ar[r, shift right=.75, "\wt{d}_2" '] & G \times Y \ar[r, "\wt{d}_1"] \ar[d]& Y \ar[d, "\pi" ] \\
\cdots & G \times G \times G \times M \ar[d, "d_0" '] \ar[r, "d_2" description] \ar[r, shift left=1.5, "d_1"] \ar[r, shift right=1.5, "d_3" ']& G \times G \times M \ar[d,"d_0" '] \ar[r, shift left=.75, "d_1"] \ar[r, shift right=.75, "d_2" '] & G \times M \ar[r, "d_1"] \ar[d, "d_0" ']& M \\
\cdots & G \times G \times M  \ar[r, "d_1" description] \ar[r, shift left=1.5, "d_0"] \ar[r, shift right=1.5, "d_2" '] & G \times M \ar[r, shift left=.75, "d_0"] \ar[r, shift right=.75, "d_1" ']  & M 
\end{tikzcd} \]

Furthermore, the section $s_0$ pulls back along each map in the following way:  $d_i^* s_0 = s_0$ for $i \geq 1$ because $d_{i+1} s_0 = s_0 d_i$.  This gives a natural identification $\wt{d}_{i+1}/G \iso \wt{d_i}$ for $i \geq 1$.  For $i=0$, we have that  $(d_0^* s_0)(g_1, g_2, \ldots g_k,x) = (g_1^{-1}, g_1, g_2, \ldots g_k, x)$.  Therefore, 
\[ (\wt{d}_{1}/G) (g_1, \ldots, g_k, x) = d_0\wt{d}_{1}(g_1^{-1}, g_1, \ldots, g_k, x) = d_0(1, g_2, \ldots, g_k, x) = (g_2, \ldots g_k, x),\]
giving $(\wt{d}_1)/G = d_0$.  Hence, when one quotients the above diagram by $G$, the result is a diagram that is naturally isomorphic to the diagram \eqref{Eq:SimpGerbeFullDiag} defining $G^\bullet \times \cL$.  
\end{proof}

\subsection{Proof of equivalence}
\begin{prop}\label{Prop:Injection}The group homomorphism $\pi_0 (G\GGrb(M))  \to \pi_0 ( \Grb(\E G\times_G M))$ is an injection.
\end{prop}
\begin{proof}
Consider an arbitrary $\cL =(Y, L, \mu) \in G\GGrb(M)$ with induced $\cEll \in \Grb( \E G \times_G M)$.  Suppose there is an isomorphism $\mathscr{K}\colon \mathscr{I}\to \cEll$, where $\mathscr{I}$ is the image of the trivial gerbe $\cI$.  Evaluating $\mathscr{K}$ on the natural map $M \to \E G \times_G M$,  described in Proposition \ref{Prop:SimplicialGerbe}, gives an ordinary non-equivariant trivialization 
\[ \cI \iso \mathscr{I}(G\times M, d_1) \xrightarrow[\iso]{\cK \= \mathscr{K}(G\times M, d_1)} \cEll(G\times M, d_1) \iso \cL  \]
in $\Grb(M)$.
As discussed in Section \ref{Subsec:IsoClasses}, we may assume $\cK$ is given by $K \in \Bun(Y)$ with an isomorphism $\upalpha \colon \delta K \to L$, which is a priori non-equivariant.  But,  $\cK$ extends to a trivialization on $\mathscr{L}(G^\bullet \times M) \iso G^\bullet \times \cL$.  Hence, the bundle $K \in \Bun(Y)$ extends to a bundle on $\Bun(Y\sslash G)$.   The well-known equivalence $\Bun(Y\sslash G) \iso G\GBun(Y)$, proven in \cite[Theorem 4.3]{Redden16b}, therefore implies that $K \in G\GBun(Y)$.  Since the isomorphism $\upalpha$ also extends to $G^\bullet \times \cL$, this implies that $\upalpha$ is equivariant with respect to the induced induced $G$-structure on $K$.  Therefore, our morphism $\cK$ is $G$-equivariant and provides a trivialization of $\cL \in G\GGrb(M)$.
\end{proof}

\begin{prop}\label{Prop:Iso}The equivariant Dixmier--Douady class gives a natural isomorphism of abelian groups 
\[ \DD_G \colon \pi_0 \left(G\GGrb(M)\right) \xrightarrow{\iso} \pi_0 \left( \Grb(\E G \times_G M) \right) \xrightarrow{\iso} H^3_G(M;\Z).\]
\end{prop}
\begin{proof}
Proposition \ref{Prop:H3GGrb} says the second map is an isomorphism. Theorem 2.8 of \cite{TuXu15} says the composition is surjective, and hence the first map is a surjection. Proposition \ref{Prop:Injection} says the first map is an injection.  Hence, the first map must also be an isomorphism.
\end{proof}

\begin{thm}\label{Thm:IsoClasses}Let $G$ be a compact Lie group.  For $M \in G\GMan$, there is a natural isomorphism of abelian groups
\[  \pi_0 \left(G\GGrbc(M) \right) \iso \HG^3(M) \]
between $\{$equivalence classes of $G$-equivariant gerbes with connection on $M\}$ and the third differential equivariant cohomology group of $M$.  Furthermore, the character diagram 
\[ \xymatrix@R=3mm@C=6mm{ 0 \ar[dr]&&&&0 \\
&H_G^{2}(M;\R/\Z) \ar[dr] \ar^{-B}[rr]&& H_G^{3}(M;\Z) \ar[ur]  \ar[dr]\\
H_G^{2}(M;\R) \ar[ur] \ar[dr]&& \HG^{3}(M) \ar[ur] \ar[dr] && H_G^{3}(M;\R)\\
& {\displaystyle \frac{\Omega^{2}_G(M)}{\Omega^{2}_G(M)_\Z}} \ar[ur] \ar_{d_G}[rr]&& \Omega^{3}_G(M)_\Z \ar[ur] \ar[dr]\\
0 \ar[ur]&&&&0,
} \] 
in which the two diagonals are both short exact sequences, is isomorphic to the following character diagram.
\[ \xymatrix@R=3mm@C=4mm{ 0 \ar[dr]&&&&0 \\
&\pi_0\left(G\GGrbc(M)_{\fl}\right) \ar[dr] \ar[rr]&& \pi_0(G\GGrb(M)) \ar[ur]  \ar[dr]\\
H_G^{2}(M;\R) \ar[ur] \ar[dr]&& \pi_0\left(G\GGrbc(M)\right) \ar[ur]_{\DD_G} \ar[dr]^{\curv_G} && H_G^{3}(M;\R) \\
& \pi_0\left( G\GGrbtrivc(M) \right) \ar^{\cI_\bullet}[ur] \ar[rr]&& \Omega^{3}_G(M)_\Z \ar[ur] \ar[dr]\\
0 \ar[ur]&&&&0
} \] 
\end{thm}
\begin{proof}The theorem follows from Proposition \ref{Prop:StienonProperties}, Proposition \ref{Prop:Iso},  and the five lemma.  Consider the commutative diagram
\[ \begin{tikzcd} G\GGrbc(M) \ar[r]  \ar[d] & G\GGrb(M) \ar[d] \\
\Grbc(\En G \times_G M) \ar[r] & \Grb( \E G \times_G M)
\end{tikzcd} \]
where the vertical maps were defined in \eqref{Eq:FundMap1} and \eqref{Eq:FundMapNoConn}, and the horizontal maps are given by forgetting connection data.  Taking isomorphism classes gives a homomorphism of short exact sequences
\[ \begin{tikzcd}[row sep=4, column sep=15]
0 \ar[r] & \pi_0(G\GGrbtrivc(M)) \ar[r] \ar[ddd, "\iso"]& \pi_0 (G\GGrbc(M)) \ar[r, "{\DD_G}"] \ar[dd]& \pi_0 (G\GGrb(M)) \ar[dd, "\iso"] \ar[r] & 0 \\ \ \\
&&\pi_0\Grbc(\En G \times_G M) \ar[r] \ar[d, phantom, "\iso" rotate=-90] & \pi_0 \Grb(\E G \times_G M) \ar[d, phantom, "\iso" rotate=-90] \\
0 \ar[r] & \frac{\O^2_G(M)}{\O^2_G(M)_\Z} \ar[r]  & \HG^3(M) \ar[r] & H^3_G(M;\Z) \ar[r]& 0.
\end{tikzcd} \]
The upper sequence was proven to be exact in Proposition \ref{Prop:StienonProperties}; the lower two vertical isomorphisms were established in Proposition \ref{Prop:H3GGrb}; the upper-right vertical isomorphism was proven in Proposition \ref{Prop:Iso}; and the lower short exact sequence is \eqref{ses2}.  Therefore, the five lemma implies the homomorphism $\pi_0(G\GGrbc(M)) \to \HG^3(M)$ is an isomorphism.  This also shows that the positively-sloped diagonals in the two character diagrams are isomorphic.

It remains to verify that the two curvature maps are compatible, in the sense that the diagram
\[ \begin{tikzcd}
 \cLh \ar[r, phantom, "\ni"] \ar[d, mapsto]& G\GGrbc(M) \ar[r, "{\curv_G}"]  \ar[d] & \O^3_G(M) \ar[d, "\iso"] \\
\cEllh \ar[r, phantom, "\ni"] & \Grbc(\En G \times_G M) \ar[r, "\curv"] & \O^3( \En G \times_G M)
\end{tikzcd} \]
commutes.  To determine $\curv(\cEllh)\in \O^3(\En G \times_G M)$, it suffices to evaluate on all possible $(P,\Theta, f) \colon X \to \En G \times_G M$.  Using the definition of $\cEllh$ in \eqref{Eq:FundMap2}, along with Proposition \ref{Prop:Quotient}, we see that
\[ \curv( \cEllh(P,\Theta,f)) = \curv(\Theta^*( f^* \cLh)) = \Theta^* (\curv_G(f^*\cLh)) =\Theta^*(f^*(\curv_G(\cLh)) \in \O^3(X). \]
This is precisely the image of $\curv_G(\cLh)$ in $\O^3(\En G \times_G M)$ when evaluated on $(P,\Theta, f)$, as defined in \eqref{Eq:EquivForms}.

Finally, note that $G\GGrbc(M)_{\fl}$ is defined as the full subgroupoid of $G\GGrbc(M)$ whose objects have curvature 0.  Thus, the sequence 
\[ 0 \to \pi_0(G\GGrbc(M))_{\fl} \to \pi_0(G\GGrbc(M)) \xrightarrow{\curv_G} \O^3(M)_\Z \to 0\]
is exact, and we use the existing short exact sequence \eqref{ses1} to define the isomorphism 
\[ H^2_G(M;\R/\Z) \xrightarrow{\iso} \pi_0(G\GGrbc(M)_{\fl}). \]
Therefore, the isomorphism $\pi_0(G\GGrbc(M)) \xrightarrow{\iso} \HG^3(M)$ induces an isomorphism between the two character diagrams.
\end{proof}

\begin{thm}\label{Thm:Equiv2Gpd}Let $G$ be a compact Lie group and $M \in G\GMan$.  The natural functors
\begin{align*} G\GGrb(M) &\overset{\simeq}\longrightarrow \Grb(\E G \times_G M), \\
 G\GGrbc(M) &\overset{\simeq}\longrightarrow \Grbc(\En G \times_G M),\end{align*}
defined in \eqref{Eq:FundMap1} and \eqref{Eq:FundMapNoConn}, are both equivalences of 2-groupoids in $\Gpdinfty$.
\end{thm}
\begin{proof}
Proposition \ref{Prop:Iso} and Theorem \ref{Thm:IsoClasses} show that both functors are bijections on isomorphism classes of objects.  As explained in Section \ref{SEC.Sheaves.of.infty.groupoids}, we must now show that both functors induce equivalences between all automorphism groupoids.  As explained in \eqref{Eq:Automorphisms}, the Picard structures naturally identify  automorphisms of a general object with  automorphisms for the trivial object, which greatly simplifies the remaining task.

Corollary \ref{Cor:Aut} gives an equivalence $\Aut(\cL) \simeq G\GBun(M)$ for any $\cL \in G\GGrb(M)$.  Similarly, the automorphism groupoid of any $\cEll \in \Grb(\E G\times_G M)$ is  equivalent to 
\[ \Aut ( \mathscr{I}) \simeq \Shv_\infty \left( \E G \times_G M,  \L(\DK(S^1 \to 0)) \right) \simeq \Shv_\infty( \E G \times_G M, \B S^1) = \Bun (\E G \times_G M). \]
It is a well-known result that the induced functor $G\GBun(M) \xrightarrow{\iso} \Bun(\E G \times_G M)$ is an equivalence of categories \cite[Theorem 4.3]{Redden16b}.  Therefore, we may conclude   that $G\GGrb(M) \xrightarrow{\simeq} \Grb( \E G \times_G M)$ is an equivalence.

Corollary \ref{Cor:Aut} also shows  $\Aut(\cLh) \simeq G\GBunc(M)_{\fl}$ for $\cLh \in G\GGrbc(M)$.  For $\cEllh \in \Grbc(\En G\times_G M)$, 
\[ \Aut(\cEllh)  \iso \Aut ( \widehat{\mathscr{I}}) \simeq \Shv_\infty \left( \En G \times_G M,  \L(\DK(S^1 \to \O^1_{\cl})) \right), \]
which is the subgroupoid of $\Shv_\infty( \En G \times_G M, \L(\DK(S^1 \to \O^1))) \simeq \Shv_\infty( \En G \times_G M, \Bn S^1)$ consisting of objects whose curvature is $0 \in \O^2(\En G \times_G M)$.  In other words,  $\Aut(\cEllh) \simeq \Bunc(\En G \times_G M)_{\fl}$.  The equivalence $G\GBunc(M) \xrightarrow{\iso} \Bunc(\En G \times_G M)$ from \cite[Theorem 4.3]{Redden16b}, when restricted to the subgroupoids of flat connections, implies 
\[ \Aut(\cLh) \simeq G\GBunc(M)_{\fl} \iso \Bunc(\En G \times_G M)_{\fl} \simeq \Aut( \cEllh).\]
\end{proof}


\section{Applications}\label{Sec:Examples}

The results from the previous section have a few immediate consequences that we briefly explain.  

\subsection{Equivariant holonomy}\label{Subsec:EquivHol}
One important general principle is the following: Any natural construction involving gerbes has an equivariant generalization, because equivariant gerbes pull back to ordinary gerbes.  More precisely, $\cLh \in G\GGrbc(M)$ determines $\cEllh \in \Grbc(\En G \times_G M)$, which pulls back along any $(P,\Theta,f) \colon X \to \En G \times_G M$ to an ordinary gerbe with connection $\cEllh(P,\Theta,f) \in \Grbc(X)$, as indicated by
\[ X \xrightarrow{(P,\Theta,f)} \En G \times_G M \xrightarrow{\cEllh} \BBn S^1.\]

A specific instance of this principle is surface holonomy for gerbes with connection.  Holonomy  for ordinary gerbe connections is defined in a number of ways, and it classifies gerbe connections up to isomorphism; see \cite{MR2077672, MR1945806, MR2648325, MR1932333} for various treatments.  If $\Sigma^2$ is a closed oriented surface, then any gerbe connection on $\Sigma^2$ is isomorphic to $\cIh_B$ for some $B\in \O^2(\Sigma^2)$.  The holonomy is defined by
\[ \Hol_{\cIh_B} \= \int_{\Sigma^2} B \mod \Z \quad \in \, \R/\Z, \]
where $S^1$ continues to be identified with $\R/\Z$ to avoid factors of $2\pi i$.  For general $\cLh \in \Grbc(M)$, the holonomy $\Hol_{\cLh}(f)$ along $f\colon \Sigma^2 \to M$ is simply $\Hol_{\cIh_B}$, where $f^*\cLh \iso \cIh_B$.  

This existing holonomy construction immediately defines a version for $G$-equivariant gerbe connections.  Suppose that $M \in G\GMan$ and $\cLh \in G\GGrbc(M)$.  Given $(P,\Theta) \in \Bunc[G](\Sigma^2)$ and a $G$-equivariant map $f\colon P \to M$, then $\Theta^*( f^* \cLh ) \in \Grbc(\Sigma^2)$.  We simply define the {\it equivariant holonomy} along $(P,\Theta, f) \colon \Sigma^2 \to \En G \times_G M$ as 
\[ \Hol_{\cLh}(P,\Theta, f) \= \Hol_{\Theta^* ( f^* \! \cLh) }  \in \R/\Z. \]
The equivariant holonomy is invariant under gauge transformations $\varphi \colon P \to P$, in that 
\[ \Hol_{\cLh}(P,\varphi^*\Theta, f\o \varphi) = \Hol_{\cLh}(P,\Theta,f).\]
Furthermore, if $W^3$ is an oriented 3-manifold, then the equivariant holonomy along the boundary of $(Q,\Theta, F) \colon W^3 \to \En G \times_G M$ is given by the integrating the image of the equivariant curvature under the Weil homomorphism,
\[ \Hol_{\cLh}  (\partial(Q,\Theta,F)) =  \int_{W^3} \curv(\Theta^*(F^* \cLh)) = \int_{W^3} \Theta^* \left(F^* \curv_G(\cLh)\right) \mod \Z.\]
Therefore, one can consider the equivariant holonomy of objects in $G\GGrbc(M)$ as a way to define  WZW terms in gauged sigma models \cite{MR1151251, MR1311654, MR2770022,MR3079005,  1410.5500}.

\begin{example}\label{Ex:TrivGrbPt}Consider $G=U(n)$, $M=\pt$, and the topologically trivial gerbe $\cI \in U(n) \GGrbc(\pt) \iso \Grbc(\Bn U(n))$.  Since $H^3(BU(n);\Z) = \O^3_{U(n)}(\pt) = 0$, the short exact sequence \eqref{ses3} induces an isomorphism
\[ H^2(BU(n);\R) \otimes \R/\Z \overset{\iso}\longrightarrow \HG[U(n)]^3(\pt), \]
indicating that there can still exist non-trivial equivariant connections.  To see these explicitly, consider a multiple of the universal first Chern form.  For any $t\in \R$, define
\[ \mathpzc{t} \= t \tfrac{i}{2\pi} \Tr (\bullet) \in (S^1 \mathfrak{u}(n)^*)^{U(n)} = \O^2_{U(n)}(\pt)  \iso H^2(BU(n);\R).\]
Then, $[\cIh_{\mathpzc{t}}] = 0 \in \HG[U(n)]^3(\pt)$ if and only if $t \in \Z$.  In fact, the equivariant holonomy along any 2-cycle $(P,\Theta) \colon \Sigma^2 \to \Bn U(n)$ is 
\[ \Hol_{\cIh_{\mathpzc{t}}} (P,\Theta) = \int_{\Sigma^2} \Theta^*\left(  t \tfrac{i}{2\pi} \Tr (\bullet) \right) = t \int_{\Sigma^2} c_1(\Theta) \mod \Z.\]

This example illustrates why the gerbe connections in this paper should be thought of as $G$-equivariant, as opposed to merely $G$-invariant.  Furthermore, it explicitly demonstrates that the category of  equivariant gerbe connections considered by Gomi in \cite{Gomi-EquivDel, Gomi-EquivGrbs} is not equivalent to  $G\GGrbc(M)$; in particular, compare \cite[Proposition 4.6 and Theorem 5.16]{Gomi-EquivDel} with Theorem \ref{Thm:IsoClasses} and the short exact sequence \eqref{ses2}.
\end{example}

\subsection{Equivariant gerbes over compact semisimple Lie groups}
Let $G$ be a finite-dimensional Lie group, viewed as a $G$-manifold under the adjoint action of $G$ on itself.  To emphasize the action is conjugation, as opposed to left or right multiplication, we write $\En G \times_{\Ad} G$ for the differential quotient stack.  Note that $(P,\Theta, g_\varphi) \colon X \to \En G \times_{\Ad} G$ is equivalent to $(P,\Theta) \in \Bunc[G](X)$, together with a gauge transformation $\varphi \colon P \to P$.  Hence,  $\cEllh \in \Grbc(\En G \times_{\Ad} G)$ associates to any such $(P,\Theta, g_\varphi)$ an ordinary gerbe with connection on $X$.  

Gerbes on $G$ are the largest source of equivariant gerbes in the existing literature.  Detailed constructions, in various forms, have been studied in \cite{BXZ-EquivGrbs, Brylinski-Grb00, MR1945806, Mein-Grb, MR3539489}.   While it is often desirable to have explicit geometric models, it is not always strictly necessary, and we now give a complimentary approach that requires very little work.

Let $\theta, \bar\theta \in \O^1(G;\fg)$ denote the left-invariant and right-invariant, respectively, Maurer--Cartan forms on $G$.  If $G$ is compact semi-simple, then $H^1(G;\R)=H^2(G;\R)=0$ \cite[Theorem 21.1]{CE48}.  Hence  $H^3(G;\R) \iso H^4(BG;\R) \iso (S^2 \fg^*)^G$, allowing us to identify any element $\eta \in H^3(G;\R)$ with $\langle,\rangle_\eta$, an $\Ad$-invariant symmetric bilinear form on $\fg$.  We normalize so that $\eta$ is canonically represented by the closed bi-invariant 3-form $\langle \theta, [\theta,\theta] \rangle_\eta \in \O^3(G)$.  This 3-form has an equivariantly closed extension \cite{MR1321064, MR1638045}; in the Cartan model this takes the form
\begin{equation}\label{Eq:EquivExtension} \omega_\eta \=  \langle \theta,[\theta,\theta] \rangle_\eta - \tfrac12 \langle \theta + \bar\theta, \bullet \rangle_\eta  \in  \O^3_G(G),\end{equation}
where $\langle \theta + \bar\theta, \bullet \rangle \in \left(S^1 \fg^* \otimes \O^1(G)\right)^G$.  If $\eta_G \in H^3_G(G;\Z)$, we will denote its image under $H^3_G(G;\Z) \to H^3_G(G;\R) \xrightarrow{\iso} H^3(G;\R)$ by $\eta$.  The last map is an isomorphism by the following lemma.

\begin{lemma}\label{Lem:H2G}If $G$ is a compact semisimple Lie group, acting on itself by conjugation, then \[ H^2_G(G;\R) = 0 \, \text{ and  } \, H^3_G(G;\R) \iso H^3(G;\R).\]
\end{lemma}
\begin{proof}Consider the standard Serre spectral sequence
\[ E_2^{i,j} = H^i(BG; H^j(G;\R)) \Rightarrow H^{i+j}_G(G;\R), \]
induced by the fibration $G \into EG\times_{\Ad} G \to BG$.  Since $G$ is compact semisimple, then $H^j(G;\R)=0$ for $0<j<3$, and $H^i(BG;\R) = 0$ for $0<i<4$.  Therefore, all terms of total degree 2 in the $E_2$-page vanish, and $H^2_G(G;\R)=0$.  The only non-zero term of degree 3 is $E^{0,3}=H^0(BG; H^3(G;\R)) \iso H^3(G;\R)$.  The construction \eqref{Eq:EquivExtension} of $\omega_\eta$ shows that every $\eta \in H^3(G;\R)$ admits an equivariant extension to $H^3_G(G;\R)$.  Hence $E^{0,3}_2 = E^{0,3}_\infty$, giving the desired isomorphism $H^3(G;\R) \iso H^3_G(G;\R)$.
\end{proof}

\begin{thm}\label{Thm:AdjointGerbe}Let $G$ be a compact semisimple Lie group.  For any $\eta_G \in H^3_G(G;\Z)$, there exists an equivariant gerbe with equivariant connection $\cLh_{\eta} \in G\GGrbc(G)$ satisfying
\[ \DD_G(\cLh_{\eta} ) = \eta_G \in H^3_G(G;\Z), \quad \curv_G (\cLh_{\eta}) = \omega_\eta \in \O^3_G(G),\] 
and such an $\cLh_\eta$ is unique up to isomorphism.
\end{thm}
\begin{proof}
Combining Theorem \ref{Thm:IsoClasses}, the short exact sequence \eqref{ses3}
\[ 0 \to \frac{H^2_G(G;\R)}{H^2_G(G;\R)_\Z} \to \HG^3(M) \xrightarrow{(\DD_G, \curv_G)} A^3_G(G) \to 0,\] 
and the fact from Lemma \ref{Lem:H2G} that $H^2_G(G;\R)=0$, we obtain isomorphisms
\[ \pi_0 \left( G\GGrbc(G) \right) \iso \HG^3(G) \iso A^3_G(G).\]
Let $\eta_G \in H^3_G(G;\Z)$.  By the definition \eqref{Eq:EquivExtension} of $\omega_\eta \in \O^3_G(G)$ and Lemma \ref{Lem:H2G}, $\eta_G$ and $\omega_\eta$ map to the same element in $H^3_G(G;\R)$.  In other words, $(\eta_G, \omega_\eta) \in A^3_G(G)$.  Therefore, there exists a unique (up to isomorphism)  $\cLh_\eta \in G\GGrbc(G)$ with the desired characteristic class and curvature.
\end{proof}

\subsection{Higher gerbes} 
Let $p\Grbc(M)$ denote the $(p+1)$-groupoid of $p$-gerbes with connection.  The cases of $p=0$ and $p=1$ correspond to $S^1$-bundles and ordinary gerbes, respectively.  Let 
\[ \B^{p+1}_{\scriptscriptstyle{\nabla}} S^1 \simeq \L \left( \DK (S^1 \xrightarrow{d} \O^1 \xrightarrow{d} \O^2 \xrightarrow{d} \cdots \xrightarrow{d} \O^{p+1})\right)  \in \Shv_\infty \]
denote the sheaf of $p$-groupoids that classifies $p$-gerbes with connection.  This defines $\B^{p+1}_{\scriptscriptstyle{\nabla}} S^1$ up to equivalence, and specific models of $p$-gerbes yield specific  constructions of $\B^{p+1}_{\scriptscriptstyle{\nabla}} S^1$.

If $M\in G\GMan$, it is now reasonable to  define $G$-equivariant $p$-gerbe connections on $M$ as
\[p\Grbc(\En G \times_G M) \= \Shv_\infty(\En G \times_G M,\B^{p+1}_{\scriptscriptstyle{\nabla}} S^1 ).\]
As seen in \eqref{Eq:EquivConns} and Theorem \ref{Thm:Equiv2Gpd}, this recovers something equivalent to the more familiar geometric notions of $G\GBunc(M)$ and $G\GGrbc(M)$ when $p=0$ and $p=1$, respectively.  Furthermore, objects in $p\Grbc(\En G \times_G M)$ have the same desirable formal properties as $\Grbc(\En G \times_G M)$.  They assign to any $(P,\Theta,f)\colon X\to \En G \times_G M$ a $p$-gerbe with connection on $X$, and to any map between principal $G$-bundles is assigned a morphism of $p$-gerbes as indicated in \eqref{Eq:Functoriality}.  Furthermore, isomorphism classes are naturally classified by the differential equivariant cohomology groups
\[ \pi_0 \left( p\text{-}\Grbc(\En G \times_G M) \right) \iso \HG^{p+2}(M),\]
which fit into the short exact sequences \eqref{ses1}--\eqref{ses3} and a character diagram analogous to Theorem \ref{Thm:IsoClasses}.

\bibliographystyle{alphanum}
\bibliography{MyBibDesk}

\newcommand{\etalchar}[1]{$^{#1}$}
\begin{thebibliography}{{Red}2}

\bibitem[AC]{MR3107517}
Camilo~Arias Abad and Marius Crainic.
\newblock Representations up to homotopy and {B}ott's spectral sequence for
  {L}ie groupoids.
\newblock {\em Adv. Math.}, 248:416--452, 2013.

\bibitem[AMM]{MR1638045}
Anton Alekseev, Anton Malkin, and Eckhard Meinrenken.
\newblock Lie group valued moment maps.
\newblock {\em J. Differential Geom.}, 48(3):445--495, 1998.

\bibitem[BCM{\etalchar{+}}]{MR1911247}
Peter Bouwknegt, Alan~L. Carey, Varghese Mathai, Michael~K. Murray, and Danny
  Stevenson.
\newblock Twisted {$K$}-theory and {$K$}-theory of bundle gerbes.
\newblock {\em Comm. Math. Phys.}, 228(1):17--45, 2002.

\bibitem[BE]{1410.5500}
Daniel Berwick-Evans.
\newblock {T}wisted equivariant elliptic cohomology from gauged perturbative
  sigma models {I}: {F}inite gauge groups, 2014.
\newblock arXiv:1410.5500.

\bibitem[BNV]{BNV}
Ulrich Bunke, Thomas Nikolaus, and Michael V\"olkl.
\newblock Differential cohomology theories as sheaves of spectra.
\newblock {\em J. Homotopy Relat. Struct.}, 11(1):1--66, 2016.

\bibitem[Bro]{MR0341469}
Kenneth~S. Brown.
\newblock Abstract homotopy theory and generalized sheaf cohomology.
\newblock {\em Trans. Amer. Math. Soc.}, 186:419--458, 1973.

\bibitem[Bry1]{Bry93}
Jean-Luc Brylinski.
\newblock {\em Loop spaces, characteristic classes and geometric quantization},
  volume 107 of {\em Progress in Mathematics}.
\newblock Birkh\"auser Boston Inc., Boston, MA, 1993.

\bibitem[Bry2]{Brylinski-Grb00}
Jean-Luc Brylinski.
\newblock {G}erbes on complex reductive {L}ie groups, 2000.
\newblock arXiv:math/0002158.

\bibitem[BSS]{MR3376591}
Marco Benini, Alexander Schenkel, and Richard~J. Szabo.
\newblock Homotopy colimits and global observables in abelian gauge theory.
\newblock {\em Lett. Math. Phys.}, 105(9):1193--1222, 2015.

\bibitem[BTW]{MR2077672}
Ulrich Bunke, Paul Turner, and Simon Willerton.
\newblock Gerbes and homotopy quantum field theories.
\newblock {\em Algebr. Geom. Topol.}, 4:407--437, 2004.

\bibitem[BV]{BV83}
Nicole Berline and Mich{\`e}le Vergne.
\newblock Z\'eros d'un champ de vecteurs et classes caract\'eristiques
  \'equivariantes.
\newblock {\em Duke Math. J.}, 50(2):539--549, 1983.

\bibitem[BXZ]{BXZ-EquivGrbs}
Kai Behrend, Ping Xu, and Bin Zhang.
\newblock Equivariant gerbes over compact simple {L}ie groups.
\newblock {\em C. R. Math. Acad. Sci. Paris}, 336(3):251--256, 2003.

\bibitem[CE]{CE48}
Claude Chevalley and Samuel Eilenberg.
\newblock Cohomology theory of {L}ie groups and {L}ie algebras.
\newblock {\em Trans. Amer. Math. Soc.}, 63:85--124, 1948.

\bibitem[CS]{CS85}
Jeff Cheeger and James Simons.
\newblock Differential characters and geometric invariants.
\newblock In {\em Geometry and topology (College Park, Md., 1983/84)}, volume
  1167 of {\em Lecture Notes in Math.}, pages 50--80. Springer, Berlin, 1985.

\bibitem[Del]{MR0498551}
Pierre Deligne.
\newblock Th\'eorie de {H}odge. {II}.
\newblock {\em Inst. Hautes \'Etudes Sci. Publ. Math.}, (40):5--57, 1971.

\bibitem[Dug]{DuggerHoCoLim}
Daniel Dugger.
\newblock {A} primer on homotopy colimits.
\newblock
  \href{http://math.uoregon.edu/~ddugger/hocolim.pdf}{http://math.uoregon.edu/$\sim$ddugger/hocolim.pdf}.

\bibitem[FH]{FH}
Daniel~S. Freed and Michael~J. Hopkins.
\newblock Chern--{W}eil forms and abstract homotopy theory.
\newblock {\em Bull. Amer. Math. Soc. (N.S.)}, 50(3):431--468, 2013.

\bibitem[FNSW]{MR2648325}
J\"urgen Fuchs, Thomas Nikolaus, Christoph Schweigert, and Konrad Waldorf.
\newblock Bundle gerbes and surface holonomy.
\newblock In {\em European {C}ongress of {M}athematics}, pages 167--195. Eur.
  Math. Soc., Z\"urich, 2010.

\bibitem[FOS]{MR1311654}
Jos{\'e}~M. Figueroa-O'Farrill and Sonia Stanciu.
\newblock Gauged {W}ess-{Z}umino terms and equivariant cohomology.
\newblock {\em Phys. Lett. B}, 341(2):153--159, 1994.

\bibitem[Fre]{Fre16}
Daniel~S. Freed.
\newblock {O}n equivariant {C}hern-{W}eil forms and determinant lines, 2016.
\newblock arXiv:1606.01129.

\bibitem[FSS]{MR3019405}
Domenico Fiorenza, Urs Schreiber, and Jim Stasheff.
\newblock \v {C}ech cocycles for differential characteristic classes: an
  {$\infty$}-{L}ie theoretic construction.
\newblock {\em Adv. Theor. Math. Phys.}, 16(1):149--250, 2012.

\bibitem[Gaw]{1512.01028}
Krzysztof Gawedzki.
\newblock {B}undle gerbes for topological insulators, 2015.
\newblock arXiv:1512.01028.

\bibitem[GL]{1512.07185}
Alexander Gorokhovsky and John Lott.
\newblock {A} {H}ilbert bundle description of differential {K}-theory, 2015.
\newblock arXiv:1512.07185.

\bibitem[Gom1]{Gomi-EquivDel}
Kiyonori Gomi.
\newblock Equivariant smooth {D}eligne cohomology.
\newblock {\em Osaka J. Math.}, 42(2):309--337, 2005.

\bibitem[Gom2]{Gomi-EquivGrbs}
Kiyonori Gomi.
\newblock Relationship between equivariant gerbes and gerbes over the quotient
  space.
\newblock {\em Commun. Contemp. Math.}, 7(2):207--226, 2005.

\bibitem[GR]{MR1945806}
Krzysztof Gawedzki and Nuno Reis.
\newblock W{ZW} branes and gerbes.
\newblock {\em Rev. Math. Phys.}, 14(12):1281--1334, 2002.

\bibitem[GS]{GuilleminSternberg99}
Victor~W. Guillemin and Shlomo Sternberg.
\newblock {\em Supersymmetry and equivariant de {R}ham theory}.
\newblock Mathematics Past and Present. Springer-Verlag, Berlin, 1999.

\bibitem[GSW1]{MR2770022}
Krzysztof Gaw{\c{e}}dzki, Rafa{\l} R. Suszek, and Konrad Waldorf.
\newblock Global gauge anomalies in two-dimensional bosonic sigma models.
\newblock {\em Comm. Math. Phys.}, 302(2):513--580, 2011.

\bibitem[GSW2]{MR3079005}
Krzysztof Gaw{\c{e}}dzki, Rafa{\l} R. Suszek, and Konrad Waldorf.
\newblock The gauging of two-dimensional bosonic sigma models on world-sheets
  with defects.
\newblock {\em Rev. Math. Phys.}, 25(6):1350010, 122, 2013.

\bibitem[HQ]{MR3335251}
Michael~J. Hopkins and Gereon Quick.
\newblock Hodge filtered complex bordism.
\newblock {\em J. Topol.}, 8(1):147--183, 2015.

\bibitem[Jef]{MR1321064}
Lisa~C. Jeffrey.
\newblock Group cohomology construction of the cohomology of moduli spaces of
  flat connections on {$2$}-manifolds.
\newblock {\em Duke Math. J.}, 77(2):407--429, 1995.

\bibitem[Kal]{MR1218928}
Jaap Kalkman.
\newblock B{RST} model for equivariant cohomology and representatives for the
  equivariant {T}hom class.
\newblock {\em Comm. Math. Phys.}, 153(3):447--463, 1993.

\bibitem[KT]{KubelThom-EDC}
Andreas K{\"u}bel and Andreas Thom.
\newblock {E}quivariant {D}ifferential {C}ohomology, 2015.
\newblock arXiv:1510.06392.

\bibitem[KV]{MR3195150}
Alexander Kahle and Alessandro Valentino.
\newblock {$T$}-duality and differential {$K$}-theory.
\newblock {\em Commun. Contemp. Math.}, 16(2):1350014, 27, 2014.

\bibitem[Lur]{MR2522659}
Jacob Lurie.
\newblock {\em Higher topos theory}, volume 170 of {\em Annals of Mathematics
  Studies}.
\newblock Princeton University Press, Princeton, NJ, 2009.

\bibitem[Mei]{Mein-Grb}
Eckhard Meinrenken.
\newblock The basic gerbe over a compact simple {L}ie group.
\newblock {\em Enseign. Math. (2)}, 49(3-4):307--333, 2003.

\bibitem[MP]{MR1932333}
Marco Mackaay and Roger Picken.
\newblock Holonomy and parallel transport for abelian gerbes.
\newblock {\em Adv. Math.}, 170(2):287--339, 2002.

\bibitem[MQ]{MathaiQuillen}
Varghese Mathai and Daniel Quillen.
\newblock Superconnections, {T}hom classes, and equivariant differential forms.
\newblock {\em Topology}, 25(1):85--110, 1986.

\bibitem[MRSV]{MRSV15j}
Michael~K. Murray, David~Michael Roberts, Danny Stevenson, and Raymond~F.
  Vozzo.
\newblock {E}quivariant bundle gerbes.
\newblock {\em Adv. Theor. Math. Phys.}, 21(4), 2017.
\newblock arXiv:1506.07931.

\bibitem[MT]{MR3681385}
Varghese Mathai and Guo~Chuan Thiang.
\newblock Differential {T}opology of {S}emimetals.
\newblock {\em Comm. Math. Phys.}, 355(2):561--602, 2017.

\bibitem[Mur]{Mur96}
M.~K. Murray.
\newblock Bundle gerbes.
\newblock {\em J. London Math. Soc. (2)}, 54(2):403--416, 1996.

\bibitem[MW]{MR3539489}
Jouko Mickelsson and Stefan Wagner.
\newblock Third group cohomology and gerbes over {L}ie groups.
\newblock {\em J. Geom. Phys.}, 108:49--70, 2016.

\bibitem[NS]{NikSch-Equiv}
Thomas Nikolaus and Christoph Schweigert.
\newblock Equivariance in higher geometry.
\newblock {\em Adv. Math.}, 226(4):3367--3408, 2011.

\bibitem[Par]{Park16a}
Byungdo Park.
\newblock {G}eometric models of twisted differential {K}-theory {I}.
\newblock {\em J. Homotopy Relat. Struct.}, 2017.
\newblock arXiv:1602.02292.

\bibitem[Red1]{Redden16a}
Corbett {Redden}.
\newblock {Differential Borel equivariant cohomology via connections.}
\newblock {\em {New York J. Math.}}, 23:441--487, 2017.

\bibitem[{Red}2]{Redden16b}
Corbett Redden.
\newblock {A}n alternate description of equivariant connections, 2016.
\newblock arXiv:1608.01297.

\bibitem[SSW]{MR2318847}
Urs Schreiber, Christoph Schweigert, and Konrad Waldorf.
\newblock Unoriented {WZW} models and holonomy of bundle gerbes.
\newblock {\em Comm. Math. Phys.}, 274(1):31--64, 2007.

\bibitem[Ste]{StevensonThesis}
Danny Stevenson.
\newblock {\em {T}he {G}eometry of {B}undle {G}erbes}.
\newblock PhD thesis, University of Adelaide, 2000.
\newblock arXiv:math.DG/0004117.

\bibitem[Sti]{Sti10}
Mathieu Sti{\'e}non.
\newblock Equivariant {D}ixmier-{D}ouady classes.
\newblock {\em Math. Res. Lett.}, 17(1):127--145, 2010.

\bibitem[TX]{TuXu15}
Jean-Louis Tu and Ping Xu.
\newblock {P}eriodic {C}yclic {H}omology and {E}quivariant {G}erbes, 2015.
\newblock arXiv:1504.08064.

\bibitem[Wal1]{Waldorf-More}
Konrad Waldorf.
\newblock More morphisms between bundle gerbes.
\newblock {\em Theory Appl. Categ.}, 18:No. 9, 240--273, 2007.

\bibitem[Wal2]{WaldorfStringCS}
Konrad Waldorf.
\newblock String connections and {C}hern-{S}imons theory.
\newblock {\em Trans. Amer. Math. Soc.}, 365(8):4393--4432, 2013.

\bibitem[Wit]{MR1151251}
Edward Witten.
\newblock On holomorphic factorization of {WZW} and coset models.
\newblock {\em Comm. Math. Phys.}, 144(1):189--212, 1992.

\end{thebibliography}

\end{document}